\documentclass[11pt,twoside, a4paper, english, reqno]{amsart} 
\usepackage{amsaddr}
\usepackage{a4wide}
\usepackage{amscd}
\usepackage{amssymb}
\usepackage{amsthm}
\usepackage{graphicx}
\usepackage{amsmath,calrsfs}
\usepackage{latexsym}
\usepackage{enumitem,dsfont}

\usepackage{upref}
\usepackage[colorlinks]{hyperref}
\usepackage{color,graphics}
\usepackage{upref,hyperref,color}

\usepackage[usenames,dvipsnames]{xcolor}
\usepackage{pdfsync}
\usepackage{ulem,cancel,pgf}

\usepackage{frcursive}
\def\k{\text{\begin{cursive}
\hspace{-0,1cm}\textcal{k}\hspace{-0,15cm}
\end{cursive}}}

\theoremstyle{plain}
\newtheorem{theorem}{Theorem}[section]
\theoremstyle{plain}
\newtheorem{lemma}[theorem]{Lemma}
\newtheorem{prop}[theorem]{Proposition}
\newtheorem{cor}[theorem]{Corollary}

\theoremstyle{definition}
\newtheorem{definition}{Definition}[section]
\newtheorem{defi}{Definition}[section]
\newtheorem{remark}{Remark}[section]

\newtheorem*{maintheorem*}{Main Theorem}
\newtheorem*{maincorollary*}{Main Corollary}

\DeclareFontFamily{U}{BOONDOX-calo}{\skewchar\font=50 }
\DeclareFontShape{U}{BOONDOX-calo}{m}{n}{
	<-> s*[1.05] BOONDOX-r-calo}{}
\DeclareFontShape{U}{BOONDOX-calo}{b}{n}{
	<-> s*[1.05] BOONDOX-b-calo}{}
\DeclareMathAlphabet{\mathcalb}{U}{BOONDOX-calo}{m}{n}
\SetMathAlphabet{\mathcalb}{bold}{U}{BOONDOX-calo}{b}{n}
\DeclareMathAlphabet{\mathbcalb}{U}{BOONDOX-calo}{b}{n}

\newcommand{\R}{\ensuremath{\mathbb{R}}}

\newcommand{\E}{\ensuremath{\mathbb{E}}}

\newcommand{\goto}{\ensuremath{\rightarrow}}

\numberwithin{equation}{section} \allowdisplaybreaks

\title[Stochastic
$3^{rd}$-grade fluids equations in	2D	and	3D]
{Invariant	measures	for a	class of	stochastic
	third	grade fluid	equations	in	$2D$	and	$3D$	bounded	domains
}

\date{\today}

\author[Yassine Tahraoui]{ Yassine Tahraoui}
\address[ Yassine Tahraoui]{
	Center for Mathematics and Applications (NovaMath),  NOVA	SST,	Portugal}
\email[Yassine Tahraoui]{tahraouiyacine@yahoo.fr}

\author[Fernanda Cipriano]{Fernanda Cipriano}
\address[Fernanda Cipriano]{
	Center for Mathematics and Applications (NovaMath), NOVA	SST and Department of Mathematics, NOVA	SST,	Portugal}
\email[Fernanda Cipriano]{cipriano@fct.unl.pt}

\allowdisplaybreaks[1]

\usepackage{comment}
\begin{document}

	\begin{abstract}
		This work aims to	investigate	the	well-posedness	and	the	existence	of	ergodic	invariant	measures	for	a	class of	third grade fluid	equations	in	bounded	domain		$D\subset\mathbb{R}^d,d=2,3,$	in	the	presence	of	a	multiplicative	noise.	First,
		we show the existence of	a	martingale	solution	by	coupling	a	stochastic	compactness	and	monotonicity	arguments.	Then,	we	prove	a	stabilty	result,		which	gives	the	pathwise	uniqueness	of	the	solution	and	therefore	the	existence	of	strong	probabilistic	solution.	Secondly,	we	use	the	stability	result	to	show	that	the	associated	semigroup	is	Feller	and	by	using	\textit{"Krylov-Bogoliubov Theorem"}	we	get	the	existence	of	an	invariant	probability	measure.	Finally,	we	show	that	all	the	invariant	measures	are	concentrated	on	a	compact	subset	of	$L^2$,		which	leads	to	the	existence	of	an	ergodic	invariant	measure.
  		\end{abstract}

	\maketitle
			\begin{footnotesize}	
			\hspace*{0.45cm}\textbf{MSC:} 35Q35,	37L40, 60H15,	 	60J25,	76A05. \\
		\end{footnotesize}

\section{Introduction}
In this work, we are concerned with the existence of		ergodic	invariant	measures
for a	class	of	 incompressible non-Newtonian fluids filling a  two or  three	dimensional   bounded domain under  Dirichlet	boundary	condition. A crucial step releys on the well-posedness of the stochastic fluid dynamic equations in order to define a convenient Markovian semigroup.	It is worth to recall that invariant measures correspond to stationary statistical solutions (or equilibrium states) which are relevant in the study of fluid flows, namely in the description and analysis of turbulent flows.

 Most studies on fluid dynamics have been devoted to Newtonian fluids, which are characterized by a linear relation between the shear stress and the strain rate	and	therefore	these fluids	are modelled by Navier-Stokes equations,	which	 has been studied	extensively from
	mathematical and physical perspectives. However,
	there exist many real	and industrial fluids  with nonlinear viscoelastic behavior  that does not obey Newton's law, and consequently  cannot be described by the classical  viscous Newtonian fluid models. These fluids include natural biological fluids such as blood, geological flows  and others, see \textit{e.g} \cite{DR95,FR80,yas-fer}	and	their	references. Therefore, it is necessary to consider  more general fluid models. 
Recently,  special attention has been devoted to the study of non-Newtonian viscoelastic fluids of differential type,	see	\textit{e.g.}	\cite{Cioran2016}.	On	the	other	hand, several simulations studies have been performed by using the third grade fluid models, in order to understand and explain the characteristics of several nanofluids,	see	\textit{e.g.} 
\cite{PP19,RHK18} and references therein,	where
 nanofluids  are  engineered colloidal suspensions of nanoparticles  in a base fluid as water, ethylene glycol and oil, which exhibit enhanced thermal conductivity compared to the base fluid, which turns out to be of  great  potential to be used in  technology	and microelectronics.  
Therefore  the  mathematical analysis of third grade fluid equations is	important	to	understand	the	behaviours	of	such	fluids.\\

Now,	Let us briefly recall how to obtain the fluid equations for non-Newtonian fluids of differential type,	for	more	details	about	Kinematics	of	such	fluids	we	refer	to	\cite{Cioran2016}.		 Denote  the velocity field	of	the fluid by	$y$  and introduce the Rivlin-Ericksen kinematic tensors   $A_n, n\geq 1$,	see	\cite{RE55},		 defined by  
\begin{align*}
A_1(y)&=\nabla y+\nabla y^T;	\,A_n(y)=\dfrac{d}{dt}A_{n-1}(y)+A_{n-1}(y)(\nabla y)+(\nabla y)^TA_{n-1}(y), \quad n=2,3,\cdots
\end{align*}

The constitutive law of fluids of grade $n$ 
reads	$
\mathbb{T}=-pI+F(A_1,\cdots,A_n), 
$
where $\mathbb{T}$ is the Cauchy stress tensor, $p$ is the pressure and 
 $F$ is an isotropic polynomial function    of degree $n$, subject to the usual	requirement of material frame indifference,	see	\textit{e.g.}	\cite{Cioran2016}.
The constitutive law of   third grade fluid	$(n=3)$ is given by the following  equation	 
\begin{align*}
\mathbb{T}=-pI+\nu A_1+\alpha_1A_2+\alpha_2A_1^2+\beta_1 A_3+\beta_2(A_1A_2+A_2A_1)+\beta_3tr(A_1^2)A_1,
\end{align*}
  where $\nu$ is the viscosity and 	$(\alpha_i)_{1,2}, (\beta_i)_{1,2,3}$ are material moduli. 		We	recall	that	the momentum equations, established by the Newton's second law, are given by
$\dfrac{Dy}{Dt}=\dfrac{dy}{dt}+y\cdot \nabla y=div(\mathbb{T}).$
If $\beta_i=0,i=1,2,3$, the constitutive equations correspond to a second grade fluids. It has been shown  that the Clausius-Duhem
inequality and the assumption that the Helmholtz free energy is a minimum in
equilibrium requires the   viscosity and   material moduli to satisfy
\begin{align}\label{secondlaw}
\nu \geq 0,\quad \alpha_1+\alpha_2=0, \quad \alpha_1\geq 0. 
\end{align}
Although second grade fluids are mathematically more treatable,	dealing  with  several non-Newtonian fluids, the rheologists have not
confirmed these restrictions \eqref{secondlaw}, thus give  the conclusion that the fluids that have been tested are not fluids of second grade but are fluids that are characterized by a different constitutive structure, we refer to \cite{FR80} and references therin for more details.  Following \cite{FR80}, in order to allow the motion of the fluid to be compatible with thermodynamic, it should be imposed that

\begin{equation}\label{third-grade-paremeters}
\nu \geq 0, \quad \alpha_1\geq 0, \quad |\alpha_1+\alpha_2 |\leq \sqrt{24\nu\beta}, \quad \beta_1=\beta_2=0, \beta_3=\beta \geq 0.
\end{equation}	
Consequently,	 the velocity field $y$  satisfies
the incompressible third grade fluid equations
\begin{align}
\label{third-grade-fluids-equations}
\begin{cases}
&\partial_t(v(y))-\nu \Delta y+(y\cdot \nabla)v(y)+\displaystyle\sum_{j=1}^dv(y)^j\nabla y^j-(\alpha_1+\alpha_2)\text{div}(A(y)^2) -\beta \text{div}[tr(A(y)A(y)^T)A(y)]\\[0.2cm]
&\qquad= -\nabla \mathbf{P} +U,\quad	\text{div}(y)=0,	\quad
v(y):=y-\alpha_1\Delta y, \quad A(y):= \nabla y+\nabla y^T,
\end{cases}
\end{align}
where the viscosity $\nu$ and 
the material moduli	 $\alpha_1,\alpha_2$, $\beta$ 	verify	\eqref{third-grade-paremeters},	  $\mathbf{P}$ denotes the pressure	and	 $U$	denotes	an	external	force. 	Notice	that	if	$\alpha_1=\alpha_2=0$	and $\beta$=0,	we	recover		the	Navier	Stokes	equations.		From mathematical point of view,  fluids of grade  $n$ constitute an hierarchy of fluids with  increasing complexity and more  nonlinear terms,  then comparing with Newtonian (grade $1$) or second grade fluids,
third grade fluids are  more complex  and require more involved analysis.\\
\\
Without	exhaustiveness,	when	$	\alpha_1>0$, the existence	of	local	solution	in	the Sobolev space $H^3$	of	the	third grade fluids equations \eqref{third-grade-fluids-equations} with Dirichlet  boundary condition were  studied  in  \cite{AC97},	see	also	\cite{SV95}.	Later on \cite{Bus-Ift-1},	the	authors	showed	the	global	existence	in	$\mathbb{R}^d,d=2,3$	for	$H^2$-valued	solution	and	uniqueness	in	2D,	we	recall	that	uniqueness	in	3D	for	$H^2$-valued	solution	is	an	open	question.	In	
 \cite{Bus-Ift-2},  
supplementing the equation 	\eqref{third-grade-fluids-equations}	with a Navier-slip boundary condition, the authors  
established the existence of a global solution for initial conditions in $H^2$
and proved that uniqueness holds in 2D.
In \cite{ Cip-Did-Gue},  the authors extended  the later deterministic results  to  stochastic setting 	in	2D. Recently, the	authors	  in \cite{yas-fer-2}  proved	the existence and uniqueness of $H^3$-local (up to a certain positive stopping time)	adapted		solution to the stochastic third grade fluids equations with Navier-slip boundary conditions	in	2D	and	3D	bounded	domain.	Let	us	refer	to	\cite{yas-fer,Tah-Cip-2}	and	their	references	for	other		questions	related	to	fluids	of	third	grade		with	Navier-boundary	conditions.\\

We emphasize that	the construction of a	solution with less regular initial data is  challenging  due to  complicated nonlinearities in	\eqref{third-grade-fluids-equations}	and	one	needs	 an additional	restriction	on	the	parameters	$\alpha_1,\alpha_2,\beta$	and	$\nu$	to	establish	some	results.	Indeed,		when	the	initial	data	belong	only	to	$H^1$	and	$\alpha_1>0$,	the	author	in	\cite{Paicu2008}	showed	the	existence	of	global	weak	solution	for	\eqref{third-grade-fluids-equations}	in	$\mathbb{R}^d,d=2,3$,	under	some	extra	restriction	on	the	parameters,	which	permits	the	application	of			a	monotonicity	techniques.	Then,		the validity of the energy equality and a weak-strong uniqueness result has been shown.	We	refer	also	to		\cite{Busuioc-Iftimie-Paicu08}	for an existence result in the stationary case,	in the presence of external forces and homogeneous Dirichlet boundary condition.		In	the	stochastic	setting,	let	us	refer	to	\cite{Nguyen-Tawri-Temam21},	where	the	authors		showed	 the global existence	of  solutions to stochastic
  		equations with a monotone operator	driven by a  L\'evy noise, including the	Ladyzenskaya-Smagorinsky type	equations	\cite{Ladyzhenskaya67}.
  		Recently,	the	authors	in	\cite{Raya-Fernanda-Yassine23}	proved	the	existence	of	weak	probabilistic	(martingale)	solution	to	\eqref{third-grade-fluids-equations}	in	the	presence	of	a	multiplicative	noise	by	coupling	monotonicity	and	stochastic	compactness	approach.	On the other hand, constructing solution with $L^2$-initial data is more challenging.	In	\cite{Paicu},	the	authors	proved	the	global	well-posedness	in	$\mathbb{R}^3$	with	free	divergence	initial	data	belongs	to	$L^2(\mathbb{R}^3)$	when	$\alpha_1=0$,	where	a	monotonocity	method	is	used	under	some	extra	restriction	on	the	parameters.	
  			Our	aim	in	this	work	is	to	consider	\eqref{third-grade-fluids-equations}	with	$\alpha_1=0$	in	the	presence	of	deterministic	external	force	$F$	and	a	stochastic	multiplicative	noise	driven by $Q$-Wiener process	$N$	\textit{i.e.}	$U=F+N$,	namely,	the	equations	read
		\begin{align}\label{equationV1}
	\dfrac{\partial	y}{\partial	t}-\nu \Delta y+(y\cdot \nabla)y-\alpha\text{div}(A(y)^2) -\beta \text{div}(|A(y)|^2A(y))=F-\nabla \textbf{P}+ N; \,\,
	\text{div}(y)=0.
	\end{align}	The		term			$\alpha\text{div}(A(y)^2)$	destroys	the	monotonocity	property	of		$-\nu \Delta y-\beta \text{div}(|A(y)|^2A(y))$	and	some	restriction	on	the	parameter	should	be	imposed,	namely	$\frac{\alpha^2}{2\nu\beta}\in]0,1[$	to	show	the	well-posedness	of	\eqref{equationV1},	see	Section	\ref{Sec2}	for	the	precise	assumptions.	In a relationship with the attempts to build fluid dynamics	models where  global well-posedness in 3D	holds, we refer to the pioneering work	\cite{Ladyzhenskaya67},	where	the	author	proposed		a	new equations to describe	the motions of viscous	incompressible fluids	with	viscosity	depends	on	the	gradient	of	the	velocity.
	Finally,	we wish to draw the reader’s attention to the fact that	\eqref{equationV1}	could also be considered as a	singular perturbation of the Navier-Stokes equations and the study of  its	singular limits is an interesting question and will be considered in future work.\\

Concerning	the	invariant	measures	in	fluids	dynamics,	many	authors	have	been	interested	in	the	subject.	Newtonian fluid dynamics has been widely studied,	without	exhaustiveness,	let	us	mention	\cite{Flan-94}	where	the	author	proved the existence of invariant measures	by using	the	dissipation	properties	for	the	2D	stochastic	Navier-Stokes	equations.	More	recently,	
the existence of an invariant measure to stochastic 2D Navier–Stokes equations in	the	presence	of	
multiplicative noise  in unbounded domains	were	proved	in	\cite{Brez2017}	by	using	$bw$-Feller	property	of	the	semigroup	associated	with	the	dynamics.	We	refer	\textit{e.g.}	to		\cite{Brez2017,Flan-94,Hairer2006}	and	their	references		about	invariant	measures	for	Newtonian	fluids.
On the other hand, the behaviors of non-Newtonian stochastic fluids are much less studied.	In	\cite{Guo2010},	the	authors	studied		the	martingale solutions and stationary solutions	for	a	stochastic non-Newtonian fluids.			The stochastic    non-Newtonian	bipolar fluid	equations	in	the	presence	of	L\'evy type	noise	were	 investigated	in	\cite{Hausenblas2016},	where	the	authors	showed	the	existence	of	unique	solution	and	an	ergodic	invariant	measure.	Finally,		the	authors	studied		in	\cite{Zhao2014}
the large time behaviors	 of solutions	to	\eqref{equationV1} in	the	deterministic	setting,	namely	with	$F=N=0$	in	$\mathbb{R}^3$.	For	the	best	of	our	knowledge,	we	are	not	familiar	with	results	about	the	investigation	of	invariant	measures	for	non-Newtonian	fluids	of	differential	type	and	our	aim	is	to	present	a	first	result	in	this	direction,	namely	the	invariant	measures	associated	with	the	dynamic	governed	by	\eqref{equationV1}.		We emphasize that the strong nonlinearities in  \eqref{third-grade-fluids-equations} make it very difficult the study of	the qualitative properties of the solution and our goal is to study similar questions for the \eqref{third-grade-fluids-equations} in a more general framework in future work.				Our aim is twofold:	first,		we	show		the	existence	and	uniqueness	of	probablistic	strong	solution			for	\eqref{equationV1}	supplemented	with	Dirichlet	boundary	conditions,	see	Theorem	\ref{exis-thm-strong}.	Then,	we	show	the	existence	of	an	ergodic	invariant	measures	in	2D	and	3D	settings,	see	Theorem	\ref{THm-invariant-measure-1}	and	Theorem	\ref{Theorem-ergodic-1}.\\

The article is organized as follows:
in Section \ref{Sec2}, we state  the equations and precise the  appropriate functional  and stochastic settings.	Then,	we	present	the	assumptions	on	the	data.  Section \ref{Sec-main-result} is devoted to the presentation of  the main results of this work. In section \ref{Sec-martingale}, we introduce an approximated system, and we prove the existence of martingale  solution by combining a stochastic compactness arguments and	monotonicity	techniques	to	deal	with	the	non	linear	terms.	Then,	we	show	a	stability	result	and	we	obtain	the	pathwise	uniqueness.	Consequently,	the	existence	of	a	strong	probabilistic	solution. Finally,	Section \ref{Section-invariant-measure} concerns the proof	of	the	existence	of	an	ergodic	invariant	measure,	under appropriate assumptions on the data.

\section{Content of the study}\label{Sec2}
Let	$\mathcal{W}$ be a cylindrical  Wiener process in	a separable Hilbert space	$\mathbb{H}$,	defined on a complete probability space	$(\Omega,\mathcal{F},P)$,  endowed with  the right-continuous filtration $\{\mathcal{F}_t\}_{t\in[0,T]}$. We assume that $\mathcal{F}_0$ contains all the P-null subset of $\Omega$ (see Subsection \ref{Noise-section} for the assumptions on the noise).
The goal  is to study the well-posedness	and	invariant	measures	of	 a	class third grade fluid. Let	$T>0$,	the fluid  fills bounded and simply connected domain $D  \subset  \mathbb{R}^d,\; d=2,3,$ with regular boundary $\partial D$, and 
its dynamics is governed by the following equations

\begin{align}\label{I}
\begin{cases}
dy=\big(F-\nabla \textbf{P}+\nu \Delta y-(y\cdot \nabla)y+\alpha\text{div}(A^2) +\beta \text{div}(|A|^2A)\big)dt+ G(\cdot,y)d\mathcal{W} \quad &\text{in } \Omega\times D \times (0,T),\\
\text{div}(y)=0 \quad &\text{in } \Omega\times D \times (0,T),\\
y=0 &\text{on } \Omega\times \partial D \times (0,T),\\
y(x,0)=y_0(x) \quad &\text{in } \Omega\times D ,
\end{cases}
\end{align}
where $y:=(y_i)_{i=1}^d$ is the velocity of the fluid, $\textbf{P}$ 
is the pressure and $F$ corresponds to the  external force.
The operator   $A$	is defined by
$ A:=A(y)=\nabla y+\nabla y^T=2\mathbb{D}(y)$.
In addition,  $\nu$ denotes  the viscosity of the fluid
and  $\alpha$, $\beta$  
are material moduli.
 The diffusion coefficient $G$ will be specified in Subsection \ref{Noise-section}.
\subsection{Notations	and		the functional setting }
Let	$T>0$,	for a Banach space $E$, we define
$$  (E)^k:=\{(f_1,\cdots,f_k): f_l\in E,\quad l=1,\cdots,k\}\;\text{ for	positive integer }	k.$$
In	the	following	$d=2,3$.
The unknowns in the system \eqref{I} are 
the velocity and the  scalar pressure random fields: \begin{align*}
y:\Omega\times D\times [0,T]&\to \mathbb{R}^d, \qquad\qquad\qquad			p:\Omega\times  D\times [0,T] \to \mathbb{R}\\
(\omega,x,t)&\mapsto (y^i(\omega,x,t))_{i=1}^d, \;\qquad(\omega,x,t)	\mapsto  p(\omega,x,t).
\end{align*}

Let	$m\in	\mathbb{N}^*$	and	$1\leq	p<	\infty$,	we	denote	by	$W^{m,p}(D)$	the	standard	Sobolev	space	of	functions	whose	weak	derivative	up	to	order	$m$	belong	to	the	Lebesgue	space	$L^p(D)$	and	set	$H^m(D)=W^{m,2}(D)$	and	$H^0(D)=L^2(D)$.
Following	\cite[Thm.	1.20	$\&$	Thm.	1.21	]{Roubicek},	we	have the continuous embeddings:
\begin{align}\label{Sobolev-embedding}
\text{ if }p<d,\quad &W^{1,p}(D) \hookrightarrow L^a(D),\ \forall a \in [1,p^*]\text{ and 
	it is  compact if }a \in [1,p^*),\nonumber
\\
\text{ if }p=d,\quad & W^{1,p}(D) \hookrightarrow L^a(D),\ \forall a <+\infty\text{ 
	is  compact}, 
\\
\text{ if }p>d,\quad &W^{1,p}(D) \hookrightarrow C(\overline D)\text{  
	is compact, }\nonumber
\end{align}
where $p^*=\frac{pd}{d-p}$ if $p<d$, 	denotes	the Sobolev embedding exponent.  Let us denote by ${\bf n}$ the exterior unit normal to the boundary $\partial D$,
and introduce the following spaces:
\begin{equation}
\begin{array}{ll}
\mathcal{V}&:=\{y \in (C^\infty_c(D))^d \,\vert \text{ div}(y)=0\},\\
H&:=\text{The	closure	of	}\mathcal{V}	\text{	in	}(L^2(D))^d	=\{ y \in (L^2(D))^d \,\vert \text{ div}(y)=0 \text{ in } D \text{ and } 
y\cdot {\bf n} =0\;\text{ on }\partial D\}, \nonumber\\[1mm]
V&:=\text{The	closure	of	}\mathcal{V}	\text{	in	}(H^1(D))^d=\{ y \in (H^1_0(D))^d \,\vert \text{ div}(y)=0 \text{ in } D  \}.
\end{array}
\end{equation}

Now, we recall the Leray-Helmholtz projector $\mathbb{P}: (L^2(D))^d \to H$, which is a linear bounded operator characterized by the following $L^2$-orthogonal decomposition
$v=\mathbb{P}v+\nabla \varphi,\;  \varphi \in H^1(D). $

Now, let us introduce  the scalar product between two matrices  $
A:B=tr(AB^T)$
and denote $\vert A\vert^2:=A:A.$
The divergence of a  matrix $A\in \mathcal{M}_{d\times d}(E)$ is given by 
$(\text{div}(A)_i)_{i=1}^{i=d}=(\displaystyle\sum_{j=1}^d\partial_ja_{ij})_{i=1}^{i=d}¨. $
The space $H$ is endowed with the  $L^2$-inner product $(\cdot,\cdot)$ and the associated norm $\Vert \cdot\Vert_{2}$. We recall that
\begin{align*}
(u,v)&=\sum_{i=1}^d\int_Du_iv_idx, \quad  \forall u,v \in (L^2(D))^d,\quad
(A,B)&=\int_D A: Bdx ; \quad  \forall A,B \in \mathcal{M}_{d\times d}(L^2(D)).
\end{align*}
On the functional space $V$,
we will consider  the following inner product
\begin{equation*}
\begin{array}{ll} 
(u,z)_V&:=(u,z)+(\nabla	u,\nabla	z),
\end{array}
\end{equation*}
and denote by $\Vert \cdot\Vert_V$ the corresponding norm. The usual norms on the classical Lebesgue and Sobolev spaces $L^p(D)$ and $W^{m,p}(D)$ will be denoted by   $\|\cdot \|_p$ and 
$\|\cdot\|_{W^{m,p}}$, respectively.
In addition, given a Banach space $E$, we will denote by $E^\prime$ its dual.
For $T>0$, $0<s<1$ and $1\leq p <\infty$, let us   recall the definition of the fractional Sobolev space  
$$W^{s,p}(0,T;E):=\{ f \in L^p(0,T;E) \; \vert \; \Vert f\Vert_{W^{s,p}(0,T;E)} <\infty\},$$
where $\Vert f \Vert_{W^{s,p}(0,T;E)}= \Big( \Vert f\Vert_{L^p(0,T;E)}^p+\displaystyle\int_0^T\int_0^T\dfrac{\Vert f(r)-f(t)\Vert_E^p}{\vert r-t\vert^{sp+1}}drdt\Big)^{\frac{1}{p}}$.\\

Since	$L^{\infty}(0,T;H)$	is	not	separable,	
it is	convenient	to	introduce	the	following	space:
$$L^2_{w-*}(\Omega;L^\infty(0,T;H))=\{ 	u:\Omega\to	L^\infty(0,T;H)	\text{	is	 weakly-* measurable		and	}	\E\Vert	u\Vert_{L^\infty(0,T;H)}^2<\infty\},$$
where	weakly-* measurable	stands	for	the	measurability	when	$L^\infty(0,T;H)$	is	endowed	with	the	$\sigma$-algebra	generated	by	the	Borel	sets	of	weak-*	topology,	see	\textit{e.g.}	\cite[Thm. 8.20.3]{Edwards}		$\&$ \cite[Rmq. 2.1]{Vallet-Zimm1}.	
For the sake of simplicity, we do not distinguish between scalar, vector or matrix-valued   notations when it is clear from the context. In particular, $\Vert \cdot \Vert_E$  should be understood as follows
\begin{itemize}
	\item $\Vert f\Vert_E^2= \Vert f_1\Vert_E^2+\cdots+\Vert f_d\Vert_E^2$ for any $f=(f_1,\cdots,f_d) \in (E
	)^d$.
	\item $\Vert f\Vert_{E}^2= \displaystyle\sum_{i,j=1}^d\Vert f_{ij}\Vert_E^2$ for any $f\in \mathcal{M}_{d\times d}(E)$.
\end{itemize}

Throughout the article,  we denote by $C,C_i, i\in \mathbb{N}$,   generic constants, which may vary from line to line.
We also consider the trilinear form 
$$
b(\phi,z,y)=(\phi\cdot \nabla z,y)=\int_D(\phi\cdot \nabla z)\cdot y dx, \quad \forall \phi,z,y \in V,$$
which verifies $b(y,z,\phi)=-b(y,\phi,z),\quad\forall  y,  z,\phi \in V$.


\subsection{The stochastic setting}\label{Noise-section}
Consider a cylindrical Wiener process $\mathcal{W}$ 
defined on the	filtred	probability	space $(\Omega,\mathcal{F},P;(\mathcal{F}_t)_t)$, which can be written as 
$\mathcal{W}(t)= \sum_{\k \ge 1} f_\k \beta_\k(t),$
where $(\beta_\k)_{\k\ge 1}$ is a sequence of  mutually independent real valued standard Wiener processes and $(f_\k)_{\k\ge 1}$ is a complete orthonormal system in a separable Hilbert space $\mathbb{H}$.
Recall that 
the sample paths  of $\mathcal{W}$ take values 
in a larger Hilbert space $H_0$ such that   $\mathbb{H}\hookrightarrow H_0$ defines a Hilbert–Schmidt
embedding. For example, the space $H_0$ can be defined as follows  
$$ H_0=\bigg\{ u=\sum_{\k \ge 1}\gamma_{\k}f_\k\;\vert\;  \sum_{ \k\geq 1} \dfrac{\gamma_\k^2}{\k^2} <\infty\bigg\}
\text{	endowed with the norm	}
 \Vert u\Vert_{H_0}^2=\sum_{ \k\geq 1} \dfrac{\gamma_\k^2}{\k^2}, \;  u=\sum_{\k \ge 1}\gamma_kf_\k.$$
Hence,  $P$-a.s. the trajectories of $\mathcal{W}$ 
belong to the space $C([0,T],H_0)$ (cf. \cite[Chapter 4]{Daprato}).\\


Next, we will precise the assumptions on the data.
\subsubsection{Definition of the diffusion coefficient and assumptions}\label{subsec-multiplicative}
Let us consider 
a family of Carath\'eodory functions
 $\sigma_\k: [0,T]\times \R^d\mapsto \R^d, \; \k \in\mathbb{N},$  
  which  satisfies
$\sigma_\k(t,0)=0$\footnote{Note that the same can be reproduced with: $\displaystyle\sum_{\k\ge 1} \Vert\sigma_\k^2(t,0)\Vert_{2}^2\in	L^1(0,T).$},  and   there exists $L > 0$  such that for a.e. $t\in (0,T)$,  
\begin{align}
\label{noise1}
\quad &\sum_{\k\ge 1}\big| \sigma_\k(t,\lambda)-\sigma_\k(t,\mu)\big|^2  \leq L  |\lambda-\mu|^2;	\quad	\forall	\lambda,\mu \in \R^d.
\end{align}

We notice that, in particular, \eqref{noise1} gives 
$\sum_{\k\ge 1}
 |\sigma_\k(t,\lambda)|^2\le L\,|\lambda|^2.
$ 

For each $t\in[0,T]$ and $H$-valued	predictable	process	$y$, we introduce the  predictable process $G(t,y)$ with values in the space of  Hilbert-Schmidt operators
$$
G(t,y): \mathbb{H}\goto (L^2(D))^d, \qquad G(t,y)f_\k= \{ x \mapsto \sigma_\k\big(t,y(x)\big)\},
\quad \k \ge 1.
$$
	 The stochastic integral 
	$\displaystyle\int_0^tG(\cdot,y)d\mathcal{W}:=\sum_{\k\ge 1}\int_0^t\sigma_\k(\cdot,y)d\beta_\k$ 
	is  well-defined	continuous $(\mathcal{F}_t)_{t}$-martingale with values in $(L^2(D))^d$.	In the sequel, given a random variable $\xi$ with values in a Polish space $E$, we  will denote by 
	$\mathcal{L}(\xi)$ its law
	$$
	\mathcal{L}(\xi)(\Gamma)=P(\xi\in\Gamma) \quad \text{for any Borel subset } \Gamma \text{ of } E.
	$$
\subsection{Assumptions}
	The	parameters  $\nu,$   $\alpha$	and $\beta$  
	 satisfy 
	\begin{align}\label{condition1}
	\nu \geq 0, \quad  \beta > 0 \quad |\alpha|\leq \sqrt{2\nu\beta}.
	\end{align}
We	recall	that	\eqref{condition1}	ensures	a	monotonicity	property	of	a	part	of	the	nonlinear	operator	of	\eqref{I}.	More	precisely,
	let	us	introduce	the	following	Banach	space	$(X,\Vert	\cdot\Vert_{X})$
$$	X=\{	u\in	(W^{1,4}(D)\cap	H_0^1(D))^d,	\quad	\text{div	}	u=0\},	\quad	\text{	where		}\Vert	\cdot\Vert_X:=\Vert	\cdot\Vert_{W^{1,4}_0}.$$
Indeed,	we	recall	that	$W^{1,4}(D)\cap	H_0^1(D)=W^{1,4}_0(D)$	endowed	with	$\Vert	\cdot\Vert_{W^{1,4}}$-norm	is		Banach	space	where		$$\Vert	u\Vert_{W^{1,4}}^4=\int_D	\vert	u\vert^4dx+\int_D\vert	\nabla	u\vert^4dx.$$
		Thanks	to	Poincaré	inequality,	see	\textit{e.g.}	\cite[Theorem	1.32]{Roubicek},	there	exists		$C_P>0$	such	that	$	\Vert	u	\Vert_{4}\leq	C_P\Vert	\nabla	u\Vert_4$	for	any	$u\in	W^{1,4}_0(D)$.	Set	$	\Vert	u	\Vert_{_{W^{1,4}_0}}=\Vert	\nabla	u\Vert_4$	then	$\Vert	\cdot\Vert_{_{W^{1,4}_0}}$	and		$\Vert	\cdot\Vert_{W^{1,4}}$	are	equivalent	norms	on	$W^{1,4}_0(D).$	Thus,	$(X,\Vert	\cdot\Vert_{W^{1,4}_0})$	is	a	Banach	space,
	as	a closed subspace 	of	$(W^{1,4}(D))^d$.
	Finally,	let	us	recall
	Korn	inequality	(see	\cite[Theorem 1.33]{Roubicek}):	there	exist	$C_K>0$	such	that
	\begin{align}\label{Korn-ineq}
	\Vert	u\Vert_{W^{1,4}_0}	\leq	C_K\Vert	A(u)	\Vert_4,\quad	\forall	u\in	W^{1,4}_0(D).	
	\end{align}
Denote	by	$\langle	\cdot,\cdot\rangle:=\langle	\cdot,\cdot\rangle_{X^\prime,X	}$
and	define	the	following	operator
\begin{align*}
T:X&\to	X^\prime\\
u&\mapsto	-\nu(1-\epsilon_0)\Delta	u-\alpha\text{ div}(A(u)^2)-\beta(1-\epsilon_0)\text{ div}(\vert	A(u)\vert^2A(u)),
\end{align*}
where	$\epsilon_0:=1-\sqrt{\dfrac{\alpha^2}{2\nu\beta}}\in]0,1[$.	Following		\cite[Lemma	2.5]{Paicu},	we	have
\begin{lemma}\label{Lemma-monotone}
	$T$	is	a	monotone	operator	\textit{i.e.}
	$	\langle	T(u)-T(v),u-v\rangle\geq0,\quad	\forall	u,v\in	X.$
\end{lemma}
Consequently,	setting	
\begin{equation}
\label{DEF_S}
S(u):=-\nu\Delta	u-\alpha\text{ div}(A(u)^2)-\beta\text{ div}(\vert	A(u)\vert^2A(u)),\end{equation} we 	obtain	the result:
\begin{cor}\label{Monotone-S-ope}	For	$\epsilon_0\in]0,1[$,
	$S$	is	a	monotone	operator	\textit{i.e.}
	$$	\langle	S(u)-S(v),u-v\rangle\geq0,\quad	\forall	u,v\in	X.$$
\end{cor}
Let	us	precise	the	assumptions	on	the	initial	data	and	the	external	forces.
\begin{itemize}
	\item[$\mathcal{H}_1:$]
	we consider $y_0: \Omega \to H$,	$F:[0,T]\to	X^\prime$ such that
	\begin{align}
	&y_0\;\text{ is }\mathcal{F}_0\text{-measurable and}\; y_0\in L^q(\Omega,H),	\quad	q>2	\text{	and	} F\in	L^{\frac{4}{3}}(0,T;X^\prime).\label{data-assumptions}
		\end{align}

	\end{itemize}

\begin{remark}
We wish to draw the reader’s attention to the fact that	$L^2(0,T;V^\prime)	\hookrightarrow	L^{\frac{4}{3}}(0,T;X^\prime).$
\end{remark}
	\section{Main	results}\label{Sec-main-result}
First,	let	us	introduce	the	notion	of	strong	solution	to	\eqref{I}.	
\begin{defi}\label{strongsol-def} We say that \eqref{I} has a strong	(pathwise) solution, if and only if there exist  a predictable process $ y: \Omega\times[0,T] \to H$	such	that:
	
	\begin{enumerate}
		\item	P-a.s. $\omega\in	\Omega$:
		$
		y(\omega,\cdot) \in  C([0,T];X^\prime)\cap L^\infty(0,T;H)
		$	and	$y(0,\cdot)= y_0$,
		\item	$y\in 	L^4(\Omega\times(0,T);X)\cap		L^2_{w-*}(\Omega;L^\infty(0,T;H))$,	
		\item	$P$-a.s. in $\Omega$ for all $t\in[0,T]$, the following equality holds
			\begin{align}\label{eqn-strong-sol}
		(	y(t),\phi)&=(y_0,\phi)+\displaystyle\int_0^t\big\langle\nu \Delta y-(	y\cdot \nabla)	y+\alpha\text{div}[A(y)^2]+\beta \text{div}[|A(	y)|^2A(	y)],\phi\big\rangle ds\nonumber\\&\quad+\displaystyle \int_0^t\langle	F,\phi\rangle	ds+\displaystyle \int_0^t\big(G(\cdot,	y),\phi\big)
		d\mathcal{W}\quad  \text{ for all } \phi \in X.
		\end{align}	  .
		
	\end{enumerate} 
	\end{defi}

\begin{remark}\label{Prog-meas} The first point of Definition \ref{strongsol-def} could be given  by replacing  $C([0,T];(W^{-1,4/3}(D))^d)\cap L^\infty(0,T;H)$	by	$C([0,T];H)$. Indeed,	the	first	point of	Definition \ref{strongsol-def}	yields $y(\omega)\in C_w([0,T];H)$\footnote{$C_w([0,T];H)$	denotes	the	Bochner space of weakly continuous functions with values in $H.$}, then, since $y$ satisfies \eqref{eqn-strong-sol}	and	(2), we have back $P$-a.s.	$\omega\in	\Omega:$	$y(\omega)\in C([0,T];H)$ according to \cite[Thm 4.2.5]{Liu-Rock}.	Moreover,	
				it	is	worth	to	mention	 that	the	point	(2)	in	Definition \ref{strongsol-def},		in	particular	$y\in	L^4(\Omega\times(0,T);(W^{1,4}(D))^d)$	allows	to	apply	infinite	dimensional	It\^o	formula	for	$\Vert	y\Vert_H^2$	in	2D	and	3D	bounded	domain,	see	\textit{e.g.}	\cite[Thm.	4.2]{Pardoux}.
\end{remark}
For the convenience of the reader,	let	
 us state the main results of our work, the detailed
 	proof of which is presented in  Sections	\ref{Sec-martingale}	and	\ref{Section-invariant-measure}.	The			first	result	is	given	in	the	following	theorem.
\begin{theorem}\label{exis-thm-strong}
	Assume that \eqref{data-assumptions} holds. Then, there exists a (strong) solution  to \eqref{I} in the sense of Definition \ref{strongsol-def}.
\end{theorem}	

\begin{proof}
		The	proof	is	devided	into	two	steps.	First,	we	construct		a	martingale	solution,	by	using	stochastic	compactness	tools	given	in	Theorem	\ref{exis-thm-mart}.	Then,		we	prove		 that		pathwise	uniqueness holds	\textit{i.e}.	$P[y_1(t)=y_2(t)]=1$	for	every	$0\leq	t\leq	T$,	where	$y_1$	and	$y_2$	are	 strong solution to \eqref{I}	with	the	same	data,	see	Corollary	\ref{cor-pathwise-uniqueness}. Consequently,	Theorem	\ref{exis-thm-strong}	follows	from	\cite[Theorem	2	\&	12.1]{Ondrejat}.
\end{proof}

The	second	main	result	of	our	work	concerns	the	existence	of		invariant	measures.
Let	$y(t;y_0),	t\geq	0$	be	the	unique	strong	solution	to	\eqref{I}.
For any 	$\varphi\in	\mathcal{B}_b(H)$\footnote{$\mathcal{B}_b(H)$	denotes	the	set	of	bounded Borel functions.},	we	define
\begin{align}\label{transition-prob-1}
(P_t\varphi)(y_0)=\E[\varphi(y(t;y_0))],	\quad	y_0\in	H,	\quad	t\geq	0.
\end{align}

\begin{theorem}\label{THm-invariant-measure-1}	Assume that  $y_0\in	H$,	$F\in	X^\prime$	and	the	coefficients		$(\sigma_{\k})_{\k\geq	1}$ 	are	independent	of	$t$,	and	satisfies	\eqref{noise1}.
	Then, there exists an invariant measure	$\mu\in\mathcal{P}(H)$\footnote{$\mathcal{P}(H)$	denotes		the	set	of	Borel	probability	measure	on	$H$.}	for	$(P_t)_t$	defined	by	\eqref{transition-prob-1},	that	is,	$P_t^*\mu=\mu$		where	$(P_t^*)_t$	denotes	 the adjoint semi-group
	acting on	$\mathcal{P}(H)$	given	by
	\begin{align}\label{semi-group-adjoint-1}
	P_t^*\mu(\Gamma)=\int_HP_t(x,\Gamma)\mu(dx)	\text{	with	}	P_t(y_0,\Gamma):=P(u(t,y_0)\in\Gamma)	\text{	for	any	}	\Gamma\in\mathcal{B}(H).
	\end{align}
\end{theorem}	
\begin{theorem}\label{Theorem-ergodic-1}
	Under	the	assumption	of	Theorem	\ref{THm-invariant-measure-1},	there exists an ergodic invariant measure $\mu$	for the transition semigroup	$(P_t)_t$,
	and	concentrated	on	$X$	satisfying	
	$	\displaystyle\int_H\Vert	x\Vert_X^4\mu(dx)	<	\infty.$
\end{theorem}
\begin{proof}	For	the	proof	of	Theorem	\ref{THm-invariant-measure-1},	see	Theorem	\ref{THm-invariant-measure}.	Theorem	\ref{Theorem-ergodic-1}	is	a	consequence	of	Proposition	\ref{prperty-invariant}	and	Theorem	\ref{Theorem-ergodic}.
\end{proof}
\section{Martingale	solution	$\&$	the	uniqueness}\label{Sec-martingale}

In the first stage, we construct a  martingale solution to \eqref{cut-off}, according to the following definition.
\begin{defi}\label{solmartingale} We say that \eqref{I} has a martingale solution, if and only if there exist a probability space $(\overline{\Omega}, \overline{\mathcal{F}},\overline{P}),$ a filtration 
	$(\overline{\mathcal{F}}_t)$, a cylindrical Wiener process 
	$\overline{\mathcal{W}} $,  $\overline{y_0} \in  L^2(\overline{\Omega},H)$ adapted  to $\overline{\mathcal{F}}_0$ and a predictable process $\overline y: \overline\Omega\times[0,T] \to H$ with a.e. paths
	\begin{align*}
	\overline y(\omega,\cdot) \in C([0,T];X^\prime)\cap L^\infty(0,T;H),
	\end{align*}
	such that 
	\begin{enumerate}
	\item		$\bar y\in 	L^4(\overline\Omega\times(0,T);X)\cap	L^2_{w-*}(\overline\Omega;L^\infty(0,T;H))$.
  \item	 $\overline	P$-a.s. in $\overline{\Omega}$ for all $t\in[0,T]$, the following equality holds
	\begin{align*}
	(\overline	y(t),\phi)&=(\overline{y}(0),\phi)+\displaystyle\int_0^t\big\langle\nu \Delta \overline	y-(\overline	y\cdot \nabla)\overline	y+\alpha\text{div}[A(\overline	y)^2]+\beta \text{div}[|A(\overline	y)|^2A(\overline	y)],\phi\big\rangle ds+\nonumber\\&+\int_0^t\langle	F,\phi\rangle	ds+\displaystyle \int_0^t\big(G(\cdot,\overline	y),\phi\big)
	d\overline{\mathcal{W}}\quad  \text{ for all } \phi \in X,	\;	\text{and
}	\mathcal{L}(\overline{ y}(0))=\mathcal{L}(y_0).
	\end{align*}	  
		\end{enumerate}
\end{defi}

Now, we are able to present the following result.
\begin{theorem}\label{exis-thm-mart}
	Assume that \eqref{data-assumptions} holds. Then, there exists a (martingale) solution  to \eqref{I} in the sense of Definition \ref{solmartingale}.
\end{theorem}
The	proof	of	Theorem	\ref{exis-thm-mart}	results	from	the	combination	of	the	following	parts.	
\subsection{Faedo–Galerkin approximation}
		Denote	by	$U:=(H^3(D))^d\cap	V$,	since	$U\underset{compact}{\hookrightarrow}	H$		we		construct	an	orthonormal	basis	in	$H$	by	using	the	eigenvectors	of	the	compact	embeeding	operator.	More	precisely,	
	there exists an orthonormal basis
	$\{e_i\}_{i\in	\mathbb{N}}$	of	$H$	such	that	$e_i\in	U$	and	
		 satisfies
	\begin{align}\label{basis1.2}
	(v,e_i)_{U}=\lambda_i(v,e_i), \quad \forall v \in U, \quad i \in \mathbb{N},
	\end{align}
	where the sequence $\{\lambda_i\}_{i\in\mathbb{N}}$ of the corresponding eigenvalues fulfils the  properties: $\lambda_i >0, \forall i\in \mathbb{N},$ and $\lambda_i \to \infty$ as $i  \to	\infty$. Note that $\{\widetilde{e}_i=\dfrac{1}{\sqrt{\lambda_i}}e_i\}$ is an orthonormal basis for $U$.	Now,	
		denote	by	$H_n=\text{span}\{e_1,\cdots,e_n\}$	and	the		operator	$P_n$	from	$U^\prime$	to	$H_n$	defined	by
	$$P_n:U^\prime\to	H_n;\quad	u\mapsto	P_nu=\displaystyle\sum_{i=1}^n\langle	u,e_i\rangle_{U^\prime,U}	e_i.	$$
	In	particular,	the	restriction	of	$P_n$	to	$H$,	denoted	by	the	same	way,	is	the	$(\cdot,\cdot)$-orthogonal	projection	from	$H$	to	$H_n$	and	given	by
	$$P_n:H\to	H_n;\quad	u\mapsto	P_nu=\displaystyle\sum_{i=1}^n 
	(	u,e_i)	e_i.	$$
	
	We 	notice	that	$\Vert	P_nu\Vert_2\leq\Vert	u\Vert_2$, $\forall u\in H$,
	then 	$\Vert	P_n\Vert_{L(H,H)}\leq	1$.
	Hence,	Lebesgue	convergence	theorem	ensures	
	$P_nu\to_n	u	$	in	$L^2(\Omega;H)$.	
	\begin{remark}\label{rmq-proj-U}
It	is	worth	to	mention	that	the	restriction	of	$P_n$	to	$U$	is	also	an	orthogonal	projection,	thanks	to	\eqref{basis1.2}	and	thus	$\Vert	P_n\Vert_{L(U,U)}\leq	1$.
	\end{remark}
	Let us consider 
	$y_{n,0}=\displaystyle\sum_{i=1}^n(y_0,e_i)e_i $
	and	set	$y_n(t)=\displaystyle\sum_{ i=1}^nc_i(t)e_i,$	$t\in[0,T]$.
	Let $M>0$ and consider a  family of  smooth  functions  $\kappa_M:[0,\infty[ \to [0,1]$  satisfying
	\begin{align}\label{cut-function}
	\kappa_M(x)=\begin{cases}
	1, \quad 0\leq x\leq M,\\[0.15cm]
	0, \quad  2M\leq x.	\end{cases}
	\end{align}
 Let us denote by $\theta_M$ the functions defined  on 
	$U^\prime$ as following 
	$$\theta_M(u)=\kappa_M (\Vert u \Vert_{U^\prime}), \quad \forall u \in  U^\prime.
	$$
	Consider	the	following	equations	\begin{align}\label{approx}
	\begin{cases}
	(y_n(t),v)&=(P_ny_0,v)+\displaystyle\int_0^t(P_nF,v)ds+\nu \displaystyle\int_0^t(P_n\Delta y_n,v)ds-\int_0^t\big(P_n	[y_n\cdot \nabla	y_n],v\big)ds\\
	&\hspace*{0.5cm}+\alpha\displaystyle\int_0^t(P_n\text{div}(A(y_n)^2),v)ds 
	+\beta\displaystyle\int_0^t(P_n\text{div}(|A(y_n)|^2A(y_n)),v)ds\\[0.15cm]&+\displaystyle\int_0^t (P_nG(\cdot,y_n),v)d\mathcal{W}\text{	for	all } v\in	H_n,	t\in[0,T]	\text{ and P-a.s.	in	}\Omega.
	\end{cases}
	\end{align}
	\begin{theorem}\label{Thm-app}
		For	each	$n\in	\mathbb{N}$,	there	exists	a	unique	predictable	process	$y_n\in	L^2(\Omega;C([0,T];H_n))$	solution	to	\eqref{approx}	satisfying	
		\begin{align}\label{estimate1}
		\E \sup_{s\in [0,T]} \Vert y_n(s)\Vert_2^2+2\nu\epsilon_0\E \int_0^{T}\Vert \nabla y_n\Vert_{2}^2dt&+\beta\epsilon_0\E\int_0^{T}\int_D|A(y_n)|^4dxdt \notag\\&\leq e^{cT} (\E\Vert	y_0\Vert_2^2+\dfrac{C}{(\beta\epsilon_0)^{\frac{1}{3}}}\int_0^T\Vert	F\Vert_{X^\prime}^{\frac{4}{3}}dt).
		\end{align}
	\end{theorem}
\begin{proof}
	For	fixed	$n\in\mathbb{N^*}$,  		consider the following approximated problem
	\begin{align}\label{cut-off}
	\begin{cases}
	\hspace*{-2cm}&(y_n^M(t)-P_ny_0,v)=\displaystyle\int_0^t(\theta_M(y_n^M)P_nF,v)ds
	+\nu \displaystyle\int_0^t(P_n\Delta y_n^M,v)ds-\int_0^t\big(P_n[\theta_M(y_n^M)y_n^M\cdot \nabla	y_n^M],v\big)ds\\
	&\hspace*{0.5cm}+\alpha\displaystyle\int_0^t(P_n\theta_M(y_n^M)\text{div}(A(y_n^M)^2),v)ds 
	+\beta\displaystyle\int_0^t(P_n\theta_M(y_n^M)\text{div}(|A(y_n^M)|^2A(y_n^M)),v)ds
		\\[0.15cm]&\hspace*{0.5cm}+\displaystyle\int_0^t (P_nG(\cdot,y_n^M),v)d\mathcal{W}\text{	for	all } v\in	H_n,	t\in[0,T]	\text{ and P-a.s.	in	}\Omega.
	\end{cases}
	\end{align}
	Set	$v=e_i,i=1,\cdots,n$	and
note	that	\eqref{cut-off}	define	a	globally Lipschitz continuous	system	of	stochastic	ODEs.	Hence,	by	using	\textit{e.g.}	”Banach fixed point theorem”,	\cite[Thm.	1.12]{Roubicek},	(see	also	\cite[Subsection	4.1]{yas-fer-2}	for similar arguments)	we	infer	the	existence	of	a	  unique predictable solution	$$y_n^M\in	L^2(\Omega;C([0,T];H_n)).$$
 Let us  define   the following  sequence  of stopping times
\begin{align*}
\tau_M^n:=\inf\{t\geq 0: \Vert y_n^M(t)\Vert_{H} \geq M \}\wedge T.
\end{align*}

	Setting
\begin{align}\label{fn}
f_n^M:= \nu \Delta y_n^M+\{-(y_n^M\cdot \nabla)y_n^M+\alpha\text{div}(A(y_n^M)^2) +\beta \text{div}(|A(y_n^M)|^2A(y_n^M))\}\theta_M(y_n^M),
\end{align}
and taking $v=e_i$ in \eqref{cut-off} for each $i=1,\cdots,n$,  we infer 
\begin{align}\label{approximation3rd}
d(y_n^M,e_i)&=(	f_n^M,e_i)	dt+(\theta_M(y_n^M)P_nF,e_i)dt+(G(\cdot,y_n^M),e_i)d\mathcal{W}\notag\\&:=(	f_n^M,e_i)	dt+(\theta_M(y_n^M)P_nF,e_i)dt+\sum_{\k\ge 1}(\sigma_\k(\cdot,y_n^M),e_i)d\beta_\k.
\end{align}
Applying It\^o's formula, we deduce
\begin{align*}
d(y_n^M,e_i)^2=2(y_n^M,e_i)(f_n^M,e_i)	dt&+2(y_n^M,e_i)(\theta_M(y_n^M)P_nF,e_i)dt\\
&+2(y_n^M,e_i)(G(\cdot,y_n^M),e_i)d\mathcal{W}+\sum_{k\geq 1} (\sigma_k(\cdot,y_n^M),e_i)^2dt.
\end{align*}
Let	 $s\in [0,\tau_M^n]$,	summing  over $i=1,\cdots,n$, we obtain
\begin{align*}
&\Vert y_n^M(s)\Vert_2^2-\Vert P_ny_{0}\Vert_2^2=2\int_0^s(	f_n^M,y_n^M)	dt+2\int_0^s(\theta_M(y_n^M)P_nF,y_n^M)dt\\&+2\int_0^s(G(\cdot,y_n^M),y_n^M)d\mathcal{W}+\int_0^s \sum_{i=1}^n\sum_{k\geq 1} (\sigma_k(\cdot,y_n^M),e_i)^2dt=J_1+J_2+J_3+J_4.
\end{align*}
After	an  integration by parts	and	using	that	$b(y_n^M,y_n^M,y_n^M)=0$, we derive
\begin{align*}
&J_1=2\int_0^s(	f_n^M,y_n^M)	dt\\
&=-\nu \int_0^s\Vert A( y_n^M)\Vert_{2}^2dt+2\alpha\int_0^s\theta_M(y_n^M)(\text{div}(A(y_n^M)^2),y_n^M)	dt\\&\qquad\qquad+ 2\beta \int_0^s\theta_M(y_n^M)(	\text{div}(|A(y_n^M)|^2A((y_n^M)),y_n^M)	dt\\
&=-\nu \int_0^s\Vert A( y_n^M)\Vert_{2}^2dt-2\alpha\int_0^s\theta_M(y_n^M)(A( y_n^M)^2,\nabla y_n^M)dt -\beta \int_0^s\theta_M(y_n^M)\int_D|A( y_n^M)|^4dxdt\\
&\leq		-2\nu \int_0^s\Vert \nabla y_n^M\Vert_{2}^2dt-\beta\int_0^s\theta_M(y_n^M)\int_D|A( y_n^M)|^4dxdt+2\vert\alpha\vert\int_0^s\theta_M(y_n^M)\Vert	A(y_n^M)\Vert_4^2\Vert\nabla	y_n^M\Vert_2.
\end{align*}
Since	$\epsilon_0=1-\sqrt{\dfrac{\alpha^2}{2\nu\beta}}\in]0,1[$,	we	get
$$2\vert\alpha\vert\int_0^s\theta_M(y_n^M)\Vert	A(y_n^M)\Vert_4^2\Vert\nabla	y_n^M\Vert_2\leq	2\nu(1-\epsilon_0)\int_0^s\Vert \nabla y_n^M\Vert_{2}^2dt+\beta(1-\epsilon_0)\int_0^s\theta_M(y_n^M)\int_D|A( y_n^M)|^4dxdt.$$
Next,		by	using	the	properties	of	the	projection	$P_n$	we	get
\begin{align*}
	J_2=2\int_0^s(\theta_M(y_n^M)P_nF,y_n^M)dt\leq	2\int_0^s\Vert	F\Vert_{X^\prime}\theta_M(y_n^M)\Vert	y_n^M\Vert_{W^{1,4}_0}	dt	\leq	2C_K\int_0^s\Vert	F\Vert_{X^\prime}\theta_M(y_n^M)\Vert	A(y_n^M)\Vert_{4}	dt,
\end{align*}
where	we	used	\eqref{Korn-ineq}.
By	using	Young	inequality,	one	has	for	any	$\delta>0$
\begin{align*}
J_2\leq	\delta\int_0^s\theta_M(y_n^M)\int_D|A( y_n^M)|^4dxdt+\dfrac{C}{\delta^{\frac{1}{3}}}\int_0^s\Vert	F\Vert_{X^\prime}^{\frac{4}{3}}dt.
\end{align*}
For	$\delta=\dfrac{\beta\epsilon_0}{2}$,	we	infer
\begin{align*}
J_2\leq	\dfrac{\beta\epsilon_0}{2}\int_0^s\theta_M(y_n^M)\int_D|A( y_n^M)|^4dxdt+\dfrac{C}{(\beta\epsilon_0)^{\frac{1}{3}}}\int_0^T\Vert	F\Vert_{X^\prime}^{\frac{4}{3}}dt.
\end{align*}
Concerning  $J_4$,  we have  
\begin{align*}
J_4=\int_0^s \sum_{i=1}^n\sum_{\k\geq 1} (\sigma_{\k}(\cdot,y_n^M),e_i)^2dt=
\int_0^s\sum_{\k\geq 1} \Vert	P_n\sigma_{\k}(\cdot,y_n^M)\Vert_2^2dt  \leq L
\int_0^s \Vert y_n^M\Vert_2^2dt.
\end{align*}
Let us estimate the stochastic term $J_3$. Let	$r\in	]0,T]$,
by using Burkholder–Davis–Gundy  and Young inequalities, there	exists	$C_{B}>0$	such	that
\begin{align*}
2\E\sup_{s\in [0,\tau_M^n\wedge	r]}\vert\int_0^s(G(\cdot,y_n^M),y_n^M)d\mathcal{W}\vert &
\leq C_{B}\E\big[\sum_{\k \ge 1}\int_0^{\tau_M^n\wedge	r}\Vert\sigma_\k(\cdot,y_n^M)\Vert_{2}^2\Vert y_n^M\Vert_2^2ds\big]^{1/2}\\
&\leq \dfrac{1}{2} \E \sup_{s\in [0,\tau_M^n\wedge	r]} \Vert y_n^M\Vert_2^2+2C_B^2L \E\int_0^{\tau_M^n\wedge	r} \Vert y_n^M \Vert_2^2dt.
\end{align*}
Hence,	with	$C(L):=	2L(1+2C_B^2)$,	we	have
\begin{align*}
&\E \sup_{s\in [0,\tau_M^n\wedge	r]} \Vert y_n(s)\Vert_2^2	+4\nu\epsilon_0 \E\int_0^{\tau_M^n\wedge	r}\Vert \nabla y_n^M\Vert_{2}^2dt+\beta\epsilon_0\E\int_0^{\tau_M^n\wedge	r}\theta_M(y_n^M)\int_D|A( y_n^M)|^4dxdt\nonumber\\&\leq 2\E\Vert y_0\Vert_2^2	+C(L)
\E\int_0^{\tau_M^n\wedge	r} \Vert y_n^M\Vert_2^2dt+\dfrac{C}{(\beta\epsilon_0)^{\frac{1}{3}}}\int_0^T\Vert	F\Vert_{X^\prime}^{\frac{4}{3}}dt.
\end{align*}
Then, the Gronwall’s inequality gives 
\begin{align*}
\E \sup_{s\in [0,\tau_M^n]} \Vert y_n^M\Vert_2^2 \leq	e^{C(L)T} (2\E\Vert	y_0\Vert_2^2+\dfrac{C}{(\beta\epsilon_0)^{\frac{1}{3}}}\int_0^T\Vert	F\Vert_{X^\prime}^{\frac{4}{3}}dt):=\mathbf{C}.
\end{align*}
Thus
\begin{align}\label{eqn-stoo-appro}
&\E \sup_{s\in [0,\tau_M^n\wedge	r]} \Vert y_n(s)\Vert_2^2	+4\nu\epsilon_0 \E\int_0^{\tau_M^n\wedge	r}\Vert \nabla y_n^M\Vert_{2}^2dt\nonumber\\&\qquad+\beta\epsilon_0\E\int_0^{\tau_M^n\wedge	r}\theta_M(y_n^M)\int_D|A( y_n^M)|^4dxdt\leq \mathbf{C}.
\end{align}
	Let us fix $n\in \mathbb{N}$. We notice that 
\begin{align*}
M^2P(\tau_M^n <T)\leq	\E(\sup_{s\in [0,\tau_M^n]} 1_{\{ \tau_M^n<T\}} \Vert y_n^M\Vert_2^2) \leq 	\E  \sup_{s\in [0,\tau_M^n]} \Vert y_n^M\Vert_2^2 \leq 	\mathbf{C}.
\end{align*}
	Thus,	there	exists	a	subset		$\tilde{\Omega}\subset	\Omega$	with	full	measure	\textit{i.e.}	$	P(\tilde{\Omega})=1$	such	that:	for	$\omega\in	\tilde{\Omega}$,	there	exists	$M_0$	verifying $\tau_M^n=T$	for	all	$M\geq	M_0$,	see	\textit{e.g.}	\cite[Theorem	1.2.1.]{Breckner}.	Since	$H\hookrightarrow	U^\prime$,	we	get		$\theta_M(u)=1$	for	all	$s\in[0,T]$	and	all	$M\geq	M_0$.	Set	$y_n=y_n^{M_0}=\displaystyle\lim_{M\to\infty}y_n^M$	with	respect	to	$H$-norm	and	notice	that	\eqref{cut-off}	becomes
				\begin{align}\label{system-no-cut}
	\begin{cases}
	(y_n(t),v)&=(P_ny_0,v)+\displaystyle\int_0^t(P_nF,v)ds+\nu \displaystyle\int_0^t(P_n\Delta y_n,v)ds-\int_0^t\big(P_n	[y_n\cdot \nabla	y_n],v\big)ds\\
	&\hspace*{0.5cm}+\alpha\displaystyle\int_0^t(P_n\text{div}(A(y_n)^2),v)ds 
	+\beta\displaystyle\int_0^t(P_n\text{div}(|A(y_n)|^2A(y_n)),v)ds\\[0.15cm]&+\displaystyle\int_0^t (P_nG(\cdot,y_n),v)d\mathcal{W}\text{	for	all } v\in	H_n,	t\in[0,T]	\text{ and P-a.s.	in	}\Omega.
	\end{cases}
	\end{align}
	Finally,	since	 $\tau_M^n \to T$ in probability, as $M\to \infty$	and 
 the sequence $\{\tau_M^n\}_M$ is monotone, the monotone convergence  theorem allows to pass to the limit in	\eqref{eqn-stoo-appro}	and deduce the	existence	of	$C,c>0$	such	that
\begin{align*}
\E \sup_{s\in [0,T]} \Vert y_n(s)\Vert_2^2&+4\nu\epsilon_0\E \int_0^{T}\Vert \nabla y_n\Vert_{2}^2dt+\beta\epsilon_0\E\int_0^{T}\int_D|A(y_n)|^4dxdt\\
&\leq	e^{cT} (\E\Vert	y_0\Vert_2^2+\dfrac{C}{(\beta\epsilon_0)^{\frac{1}{3}}}\int_0^T\Vert	F\Vert_{X^\prime}^{\frac{4}{3}}dt).
\end{align*}
\end{proof}
\subsection{Tightness}\label{compactness}
From	Theorem	\ref{Thm-app}	and	
	\cite[Lemma	2.1]{Flan-Gater},	we	get
\begin{align}\label{frac-noise}
(P_n\int_0^\cdot	G(\cdot,y_n)d\mathcal{W})_n	\text{	is	bounded	in	}	L^2(\Omega;W^{\eta,2}(0,T;(L^2(D))^d))	\text{	for	}	\eta<\dfrac{1}{2}.	
\end{align}	
The	Korn	inequality	(see	\textit{e.g.}	\cite[Thm.	1.33]{Roubicek})		and	
\eqref{estimate1}	ensure	that	$(y_n)_n$	is	bounded	in	$L^4(\Omega\times(0,T);(W^{1,4}_0(D))^d)$.	Moreover,	we	have

\begin{lemma}\label{lemma-boun-dual}
Let $y_n$ be a solution to equation \eqref{approx} given by  Theorem \ref{Thm-app}, and $S$ as defined in 
\eqref{DEF_S}. Then
	\begin{enumerate}
		\item	$(S(y_n))_n$	is	bounded	by	$K>0$	in	$L^{4/3}(\Omega;L^{4/3}(0,T;X^\prime))$.
		\item	
		$\left(\partial_t(y_n-P_n\int_0^\cdot		G(\cdot,y_n)d\mathcal{W})\right)_n$	is	bounded	by	$K>0$	in	$L^{4/3}(\Omega;L^{4/3}(0,T;U^\prime))$.
			\end{enumerate}

\end{lemma}
\begin{proof}
	By	using	\eqref{approx},	we	write P-a.s.	in	$\Omega$
			\begin{align*}
	\begin{cases}
	\partial_t(y_n-P_n\int_0^\cdot	G(\cdot,y_n)d\mathcal{W})=P_nF+\nu P_n\Delta y_n-P_n	[y_n\cdot \nabla	y_n]+\alpha	P_n\text{div}(A(y_n)^2)+ 
	\beta	P_n\text{div}(|A(y_n)|^2A(y_n)),\\[0.15cm]
	y_n(0)=P_ny_0,\quad	\text{div	}y_n=0.
	\end{cases}
	\end{align*}

First,	let	$h\in	U^\prime$	and	note	that	
$\Vert	P_n	h\Vert_{U^\prime}\leq	\Vert	h\Vert_{U^\prime}$
thanks	to	Remark	\ref{rmq-proj-U}.	In	addition,	since	$X^\prime	\hookrightarrow	U^\prime$	there	exists	$C>0$	such	that	$\Vert	k\Vert_{U^\prime}\leq	C\Vert	k\Vert_{X^\prime}$	for	any	$k\in	X^\prime.$	Thus,	to	prove	Lemma	\ref{lemma-boun-dual},	it	is	enough	to	show	that	all	the	terms	in	$(\mathcal{A}_n)_n$	is	bounded	in	$L^{4/3}(\Omega;L^{4/3}(0,T;X^\prime))$,	where	\begin{align*}
	\mathcal{A}_n=F+\nu \Delta y_n-y_n\cdot \nabla	y_n+\alpha	\text{div}(A(y_n)^2)+ 
	\beta\text{div}(|A(y_n)|^2A(y_n)).
\end{align*}
Indeed,		we	recall	that	$F\in	L^{\frac{4}{3}}(0,T;X^\prime)$.	Regarding 	the other	terms, 	there	exists	$C>0$	such	that
	\begin{align*}
	\E\int_0^T\Vert	\Delta y_n\Vert_{X^\prime}^{4/3}dt&\leq	C	\E\int_0^T\Vert\nabla y_n\Vert_{L^{4/3}(D)}^{4/3}dt\\&\leq	C(D,T)(\E\int_0^T\Vert\nabla y_n\Vert_{L^{2}(D)}^2dt)^{2/3}\leq	C_*.
	\end{align*}
	Next,	by	using	interpolation	inequality	and	that	$W^{1,4}(D)\hookrightarrow	L^\infty(D)$, we deduce 
\begin{align*}
\E\int_0^T\Vert y_n\cdot \nabla	y_n
\Vert_{X^\prime}^{4/3}&\leq	C	\E\int_0^T\Vert y_n\Vert_{L^{8/3}(D)}^{8/3}dt\leq	C	\E\int_0^T\Vert y_n\Vert_{L^{2}(D)}^{2/3}\Vert y_n\Vert_{L^{\infty}(D)}^2dt\\&	\leq	C(D)\E\int_0^T\Vert y_n\Vert_{L^{2}(D)}^{2/3}\Vert 	y_n\Vert_{W^{1,4}(D)}^2dt\\
&	\leq	C(D)\E\int_0^T\Vert y_n\Vert_{L^{2}(D)}^{4/3}dt+C(D)\E\int_0^T\Vert 	y_n\Vert_{W^{1,4}(D)}^4dt\\
&\leq	C(D)\E\int_0^T\Vert y_n\Vert_{L^{2}(D)}^{2}dt+C(D,T)+C(D)\E\int_0^T\Vert 	y_n\Vert_{W^{1,4}(D)}^4dt	\leq	C_*.
\end{align*}
After	an	integration	by	parts,	we	get
\begin{align*}
\E\int_0^T\Vert	\text{div}(A(y_n)^2)\Vert_{X^\prime}^{4/3}dt&\leq	C
\E\int_0^T\Vert	A(y_n)^2\Vert_{L^{4/3}(D)}^{4/3}dt\leq	C
\E\int_0^T\Vert	\nabla	y_n\Vert_{L^{8/3}(D)}^{8/3}dt\\
&\leq	C
\E\int_0^T\Vert	\nabla	y_n\Vert_{L^{4}(D)}^{4}dt+C(T,D)\leq	C_*,\\
\E\int_0^T\Vert	\text{div}(\vert	A(y_n)\vert^2A(y_n))\Vert_{X^\prime}^{4/3}dt&\leq	C
\E\int_0^T\Vert	A(y_n)\Vert_{L^{4}(D)}^4dt
\leq	C_*,
\end{align*}
	where	we	used	that	$(	y_n)_n$	is	bounded	in	$L^4(\Omega;L^4(0,T;(W^{1,4}(D))^d))\cap	L^2(\Omega;L^\infty(0,T;H))$.	
\end{proof}
The following lemma is proposed to gather the previous estimates.
\begin{lemma}\label{lemma-tight}	Let	$T>0$,	there	exists	$K>0$	independent	of	$n$	such	that
	\begin{enumerate}
	\item	$(y_n)_n$	is	bounded	by	$K$	in	$L^2(\Omega\times	(0,T);V)\cap	L^4(\Omega\times(0,T);(W^{1,4}_0(D))^d)$.
	\item	$(y_n)_n$	is	bounded	by	$K$	in	$L^2(\Omega;C([0,T];H))$.
	\item	$(P_n\int_0^\cdot	G(\cdot,y_n)d\mathcal{W})_n$		is	bounded		by	$K$	in		$L^2(\Omega;W^{\eta,2}(0,T;(L^2(D))^d))	\text{	for	any	}	\eta<\dfrac{1}{2}.$
	\item	$(y_n-P_n\int_0^\cdot	G(\cdot,y_n)d\mathcal{W})_n$	is	bounded	by	$K$	in	$L^2(\Omega\times	(0,T);(L^2(D))^d)$	\\
and		$\partial_t(y_n-P_n\int_0^\cdot	G(\cdot,y_n)d\mathcal{W})_n$	is	bounded	by	$K$	in	$L^{4/3}(\Omega;L^{4/3}(0,T;X^\prime))$.
\item	$(S(y_n))_n$	is	bounded	by	$K>0$	in	$L^{4/3}(\Omega;L^{4/3}(0,T;X^\prime))$.

	\end{enumerate}
\end{lemma}
Let	us	define	the	space
$$	\mathbb{W}=\{v:	v\in	L^2(0,T;H),	\quad	\partial_t	v\in	L^{4/3}(0,T;X^\prime)	\}.$$
From	Lemma	\ref{lemma-tight}$_{(4)}$,
	we	know that	$(y_n-P_n\int_0^\cdot		G(\cdot,y_n)d\mathcal{W})_n$	is bounded	in	$L^{4/3}(\Omega;\mathbb{W})$.	On	the	other	hand,	note	that	$L^{4/3}(\Omega;\mathbb{W})\hookrightarrow	L^{4/3}(\Omega;W^{\eta,4/3}(0,T;X^\prime))$  for $0<\eta\leq 1$.	By	using  the expression		
	$$y_n=y_n-P_n\int_0^\cdot	G(\cdot,y_n)d\mathcal{W}+P_n\int_0^\cdot		G(\cdot,y_n)d\mathcal{W}$$	and	Lemma	\ref{lemma-tight}$_{(3)}$,	we	obtain			the	 next	result
\begin{cor}\label{cor-tight}
	$(y_n)_n$	is	bounded	by	 a positive constant $K$	in	$L^{4/3}(\Omega;W^{\eta,4/3}(0,T;X^\prime))$	for	any	$0<\eta<\dfrac{1}{2}.$
\end{cor}
 Furthermore, we	have	the	following	result:			
\begin{lemma}\label{Lemma-Aldous}
	Let		$(\tau_n)_{n\in\mathbb{N}}$	be	a	sequence		of	$(\mathcal{F}_t)_{t\in[0,T]}$-stopping	times	with	$\tau_n\leq	T$.	Then,
	\begin{align}
	\forall	\theta>0	\quad\forall	\eta>0	\quad\exists	\delta>0	\text{	such	that	}	\displaystyle\sup_{n\in	\mathbb{N}}\sup_{0\leq	\epsilon\leq	\delta}P(\Vert	y_n(\tau_n+\epsilon)-y_n(\tau_n)\Vert_{U^\prime}	\geq	\eta)\leq	\theta.
	\end{align}
\end{lemma}
\begin{proof}
	Taking 	$0\leq	s\leq	t\leq	T$,		we	have
	\begin{align*}
	y_n(t)-y_n(s)&=\int_s^t[P_nF+\nu P_n\Delta y_n-P_n	[y_n\cdot \nabla	y_n]+\alpha	P_n\text{div}(A(y_n)^2)+ 
	\beta	P_n\text{div}(|A(y_n)|^2A(y_n))]ds\\
	&\qquad+P_n\int_s^t	G(\cdot,y_n)d\mathcal{W}=I_1^n(s,t)+I_2^n(s,t)
	\end{align*}
	Let	$(\tau_n)_{n\in\mathbb{N}}$	be	a	sequence	of	stopping	times	such	that	$0\leq	\tau_n\leq	T$	and	$\epsilon>0$.
	By	using	Holder	inequality	and	Lemma	\ref{lemma-tight},	there	exists	$C>0$	such	that	
	\begin{align}\label{Aldous1}
	\E	\Vert	I_1^n(\tau_n,\tau_n+\epsilon)\Vert_{U^\prime}\leq	C	\E	\Vert	I_1^n(\tau_n,\tau_n+\epsilon)\Vert_{X^\prime}\leq	CK	\epsilon^{1/4}.
	\end{align}
	Concerning	$I_2^n(s,t)$,	by	using	\eqref{noise1}	
	\begin{align}
	\E	\Vert	I_2^n(\tau_n,\tau_n+\epsilon)\Vert_{U^\prime}^2&=\sum_{\k\geq 1} \E		\int_{\tau_n}^{\tau_n+\epsilon}\Vert	P_n\sigma_{\k}(\cdot,y_n)\Vert_{U^\prime}^2	ds\notag \\
	&\leq	C\sum_{\k\geq 1} \E		\int_{\tau_n}^{\tau_n+\epsilon}\Vert	\sigma_{\k}(\cdot,y_n)\Vert_{2}^2	ds\leq	CL	 \E		\int_{\tau_n}^{\tau_n+\epsilon}\Vert	y_n\Vert_{2}^2ds	\leq	C\epsilon,	\label{Aldous2}
	\end{align}
	thanks	to	Lemma	\ref{lemma-tight}$_{(2)}$.	Let	$\eta>0$	and	$\theta>0$,	by	using	\eqref{Aldous1},	we	infer that
	\begin{align}
	P(\Vert	I_1^n(\tau_n,\tau_n+\epsilon)\Vert_{U^\prime}	\geq	\eta)	\leq	\dfrac{1}{\eta}\E	\Vert	I_1^n(\tau_n,\tau_n+\epsilon)\Vert_{U^\prime}\leq		\dfrac{CK	\epsilon^{1/4}}{\eta},	\quad	n\in	\mathbb{N}.
	\end{align}
	Set	$\delta_1:=(\dfrac{\eta}{CK}\theta)^{4}$.	Then,	we	get
	$	\displaystyle\sup_{n\in	\mathbb{N}}\sup_{0\leq	\epsilon\leq	\delta_1}P(\Vert	I_1^n(\tau_n,\tau_n+\epsilon)\Vert_{U^\prime}	\geq	\eta)\leq	\theta.	$
	On	the	other	hand,	by	using	\eqref{Aldous2}	
	\begin{align}
	P(\Vert	I_2^n(\tau_n,\tau_n+\epsilon)\Vert_{U^\prime}	\geq	\eta)	\leq	\dfrac{1}{\eta^2}\E	\Vert	I_2^n(\tau_n,\tau_n+\epsilon)\Vert_{U^\prime}^2\leq		\dfrac{C	\epsilon}{\eta^2},	\quad	n\in	\mathbb{N}.
	\end{align}
	Setting	$\delta_2:=\dfrac{\eta^2}{C}\theta$,		we obtain	$	\displaystyle\sup_{n\in	\mathbb{N}}\sup_{0\leq	\epsilon\leq	\delta_2}P(\Vert	I_1^n(\tau_n,\tau_n+\epsilon)\Vert_{U^\prime}	\geq	\eta)\leq	\theta,	$
	which	completes	the	proof.\end{proof}
Define	\begin{align}\label{space-tightness-sol}
\mathbf{Z}:=C([0,T];U^\prime)\cap	C([0,T];H_{weak})\cap	L^2(0,T;H),
\end{align}
where	$E_{weak}$	represents a Banach	space	$E$ endowed with the weak topology.
In addition, we denote  	by	$(\mathbf{Z},\mathcal{T})$	the	topological	product space	$\mathbf{Z}$	endowed	with		the	supremum	of	the	corresponding	topologies	$\mathcal{T}$.
Let	us	introduce  the following space
$$	\mathbf{Y}:=C([0,T]; H_0)\times	\mathbf{Z}\times H$$
Denote by $\mu_{y_n}$ the law of $y_n$  on $\mathbf{Z}$,  $\mu_{y_0^n}$ the law of $P_ny_0$ on $H$,  and $\mu_{\mathcal{W}}$ the law of $\mathcal{W}$ on $C([0,T]; H_0)$ and their joint law on $\mathbf{Y}$ by $\mu_n$.
\begin{lemma}\label{tight-force} The set	$\{ \mu_{y_0^n}; n\in \mathbb{N}\}$  is tight on   $H$. 
\end{lemma}
\begin{proof} 
	 We know that  $P_ny_0$ converges strongly to $y_0$ in $L^2(\Omega; H)$. Since $H$ is separable Banach space, from Prokhorov theorem,  for any $\epsilon >0$, there exists a compact set $K_\epsilon \subset	H$ such that $$\mu_{y_0^n}(K_\epsilon)=P(P_ny_0 \in K_\epsilon) \geq 1-\epsilon.$$
\end{proof}
 Taking into account that the law $\mu_{\mathcal{W}}$  is a Radon measure on  $C([0,T]; H_0)$,	we	obtain
\begin{lemma}\label{tight}
	The set  $\{\mu_{\mathcal{W}}\}$ is tight on   $C([0,T]; H_0)$.
\end{lemma}
Lemma	\ref{Lemma-Aldous}	ensures	that	$(y_n)_n$	satisfies				\cite[Condition	(\textbf{A});	Definition	3.7]{Brez2013})	in	the	space
	$ C([0,T];U^\prime)$.
Moreover,	thanks	to	Lemma	\ref{lemma-tight},	Lemma	\ref{Lemma-Aldous}	and	\cite[Corollary	3.9]{Brez2013},	we	get
\begin{lemma}\label{tight-2}
	The set $\{ \mu_{y_n}; n\in \mathbb{N}\}$ is tight on $(\mathbf{Z},\mathcal{T}).$
\end{lemma}

As a conclusion, we have the following corollary: 
\begin{cor}\label{comapct-law}
	The set of joint law $\{\mu_n; n\in \mathbb{N}\}$ is tight on $\mathbf{Y}$.
\end{cor}

\subsection{Subsequence extractions}\label{sub-extraction}
By	using Corollary \ref{comapct-law}	and	Jakubowski’s version of the Skorokhod Theorem	in	non	metric	spaces,	\cite[Theorem	2]{Jakubowski97}	(see	also	\cite[Corollary	3.12]{Brez2013}),	we	 can extract a  subsequence   $(n_k)_{k\in\mathbb{N}}$ 
such	that	the	following	lemma	holds.
\begin{lemma}\label{skorohod-cv} There exists a probability space $(\overline{ \Omega}, \overline{\mathcal{F}},\overline	P)$, and a family of $\mathbf{Y}$-valued random variables $\{ (\overline{\mathcal{W}}_k, \overline{ y_k},  \overline{y_0^k}) , k \in \mathbb{N}\}$  and $\{( \mathcal{W}_\infty, y_\infty,  \bar y_0)\}$  defined  on $(\overline{ \Omega}, \overline{\mathcal{F}},\overline	P)$ such that
	\begin{enumerate}
		\item $\mu_{n_k}=\mathcal{L}( \overline{\mathcal{W}}_k, \overline{y_k},  \overline{y_0^k} ), \forall k \in \mathbb{N}$;
		\item $( \overline{\mathcal{W}}_k, \overline{y_k},  \overline{y_0^k} )$ converges to $(  \mathcal{W}_\infty, y_\infty,\bar y_0)$ $\overline P$-a.s. in $\mathbf{Y}$;
	\end{enumerate}
\end{lemma}
For the sake of clarity, the expectation with respect to $(\overline{ \Omega}, \overline{\mathcal{F}},\overline	P)$ will be denoted by $\overline\E$.	Now,	let	us	present	some	results	in	order	to	pass	to	the	limit	in	the	stochastic	integral.

\begin{definition}\label{def-filtration}
	For $t\in [0,T]$ and $k\in	\mathbb{N}$, we define $\overline{\mathcal{F}}^{k' }_t$ to be the smallest sub $\sigma$-field of 
	$\overline{\mathcal{F}}$ generated by $\overline{\mathcal{W}}_k(s),	\overline{y}_k(s)$ for $0\leq s\leq t$ and $\overline{y_0^k}$. The right-continuous, $\overline{P}$-augmented filtration of $(\overline{\mathcal{F}}^{k'}_t)_{t\in [0,T]}$, denoted by $(\overline{\mathcal{F}}^{k}_t)_{t\in [0,T]}$ is defined by
	\[\overline{\mathcal{F}}^{k}_t:=\bigcap_{T\geq s>t}\sigma\left[\overline{\mathcal{F}}^{k '}_s\cup \{\mathcal{N}\in\overline{\mathcal{F}} \, : \, \overline{\mathds{P}}(\mathcal{N})=0)\}\right].\]
\end{definition}
Since	$\mathcal{L}(\overline{\mathcal{W}}_k)=\mathcal{L}(\mathcal{W})$,
by	using	the	same	arguments	used	in	\cite[Lemma	2.3]{NYGA},	we	obtain
\begin{lemma}\label{Lemma-Wiener-k}
	$\overline{\mathcal{W}}_k$	is	$Q$-Wiener process with values in the separable Hilbert space $H_0$	where	$Q=\text{diag}(\dfrac{1}{n^2}),	n\in	\mathbb{N}^*$,	and	$Q^{1/2}(H_0)=\mathbb{H}$	with	respect	to	the		filtration	$\overline{\mathcal{F}}^{k}_t$.
\end{lemma}	
As	a	consequence,	note	that	$\int_0^tG(s,\overline{	y_k}(s))d\overline{\mathcal{W}}_k(s)$	is	well-defined	It\^o	integral.	Now, we want to recover the stochastic integral and our system on the new probability	space.
Thanks	to	the	equality	of	laws,	see	Lemma	\ref{skorohod-cv}$_{(1)}$,		and	by	using	a	similair	arguments	used	in	\cite[Subsection	4.3.4]{Bensoussan95},	we	are	able	to	infer
\begin{lemma}\label{eqn-new}
	For	any	$t\in[0,T]$	and $\overline{P}$-a.s.	in	$\overline\Omega$,	
	for	all  $i=1,\cdots,k$
	\begin{align*}
	(\int_0^{t} G(\cdot,\overline{y_k}))\,d\overline{\mathcal{W}}_k,e_i)&=(\overline{y_k}(t)-\overline{y_0^k},e_i)+\int_0^t(P_nF,e_i)ds\\
	&\qquad- \displaystyle\int_0^t(\nu\Delta \overline{y_k}-\overline{y_k}\cdot \nabla	\overline{y_k}+\alpha\text{div}(A(\overline{y_k})^2)+\beta\text{div}(|A(\overline{y_k})|^2A(\overline{y_k})),e_i)ds. 
	\end{align*}
\end{lemma}
Let	$(\overline{\mathcal{F}}^{\infty}_t)_{t\in [0,T]}$	be	the	$\overline{P}$-augmented filtration of	$\sigma(\mathcal{W}_\infty(s),	y_\infty(s),	\overline{y_0};0\leq s\leq t)$. 
\begin{lemma}\label{Wienr-limit}
	$\overline{\mathcal{W}}_k$	converges	to	$\mathcal{W}_\infty$	in	$L^2(\overline{\Omega},C([0,T]; H_0))$	and
	$\mathcal{W}_\infty=(\mathcal{W}_\infty(t))_{t\in [0,T]}$ is a $H_0$-valued, square integrable $(\overline{\mathcal{F}}^{\infty}_t)_{t\in [0,T]}$-martingale with quadratic variation process $tQ$ for any $t\in [0,T]$. 
\end{lemma}
\begin{proof}
	Let	$p>2$,	note	that	$$\displaystyle\overline{\E}\sup_{s\in [0,T]}\Vert	\overline{\mathcal{W}}_k(s)\Vert_{H_0}^p=\displaystyle	\E\sup_{s\in [0,T]}\Vert	\mathcal{W}(s)\Vert_{H_0}^p\leq	C(T\sum_{ n= 1}^\infty\dfrac{1}{n^2})^{p/2},$$
	where	$C>0$	 is independent	of	$k$	from	BDG	inequality.	Thus,	Vitali's theorem	and	Lemma	\ref{skorohod-cv}$_{(2)}$	ensures	the	convergence	in	$L^2(\overline{\Omega},C([0,T]; H_0))$.	The	rest	of	the	lemma	is	a	consequence	of	Lemma	\ref{skorohod-cv},	we	refer	\textit{e.g.}	to	\cite[Subsection	2.4	]{Vallet-Zimm1}	for	detailed	and	similair	arguments.
\end{proof}

We	recall	that	$y_n\in	C([0,T];H_n)$	P-a.s.,	since	$y_{n_k}$	and	$\overline{	y_k}$	have	the	same	laws,	and	$C([0,T];H_{n_k})$	is	a	Borel	subset	of	$C([0,T];U^\prime)\cap	C([0,T];H_{weak})\cap	L^2(0,T;H),$	one	has
\begin{align}\label{lawyk}
\mathcal{L}(\overline{y_k})[C([0,T];H_k)]=1.
\end{align}
Similarly	to	Lemma	\ref{lemma-tight}	and	by	using	the	equality	in	laws,
we	are	able	to	infer	the	following.

\begin{lemma}\label{approx-martingale}	
	Let	$T>0$,	there	exists	a	unique	predictable	solution	$\overline{	y_k}	\in	C([0,T];H_k)$	such	that
	\begin{align}\label{eqn-martingale-appr}
	\overline{	y_k}(t)=\overline{y_0^k}&+	 \displaystyle\int_0^t(P_kF+\nu	P_k\Delta \overline{	y_k}-P_k\overline{	y_k}\cdot \nabla	\overline{	y_k}+\alpha	P_k\text{div}(A(\overline{	y_k})^2)+\beta	P_k	\text{div}(|A(\overline{	y_k})|^2A(\overline{	y_k})))ds\nonumber\\
	&\quad+\int_0^{t} P_kG(\cdot,\overline{	y_k}))\,d\overline{\mathcal{W}}_k,\quad	\forall	t\in[0,T], 
	\end{align}
	with	respect	to	new	stochastic	basis	$(\overline{\Omega},\overline{\mathcal{F}},\overline{P};(\overline{\mathcal{F}}^{k}_t)_{t\in [0,T]})$.	Moreover,
	there	exists	$K>0$	independent	of	$k$	such	that
	\begin{enumerate}
		\item	$(\overline{y_k})_k$	is	bounded	by	$K$	in	$L^2(\overline\Omega\times	(0,T);V)\cap	L^4(\overline\Omega\times(0,T);(W^{1,4}(D))^d)$.
		\item	$(\overline{y_k})_k$	is	bounded	by	$K$	in	$L^2(\overline\Omega;C([0,T];H))$.
		\item	$(\overline{y_k}-\int_0^\cdot	P_k	G(\cdot,\overline{y_k})d\overline{\mathcal{W}}_k)_k$	is	bounded	by	$K$	in	$L^2(\overline\Omega\times	(0,T);(L^2(D))^d)$	\\
		and		$\partial_t(\overline{y_k}-\int_0^\cdot	P_k	G(\cdot,\overline{y_k})d\overline{\mathcal{W}}_k)_k$	is	bounded	by	$K$	in	$L^{4/3}(\overline\Omega;L^{4/3}(0,T;X^\prime))$.
		\item	$(S(\overline{y_k}))_k$	is	bounded	by	$K>0$	in	$L^{4/3}(\overline\Omega;L^{4/3}(0,T;X^\prime))$.
		
	\end{enumerate}
\end{lemma}
\begin{remark}
	Thanks	to	the	uniqueness	of	the	solution,	$(\overline{\mathcal{F}}^{k}_t)_{t\in [0,T]}$	can	be	chosen	independently	of	$\overline{y_k}$	(see	\textit{e.g.}	\cite[Lemma	2.6]{NYGA}).
\end{remark}
\subsection{Proof	of	Theorem	\ref{exis-thm-mart}}

We	will	prove		Theorem	\ref{exis-thm-mart}	in	two	steps.
\subsubsection{Step	1:}
Thanks	to	Subsection	\ref{sub-extraction},	we	obtain\begin{lemma}\label{lemma-cv-limit}
	There	exist	$\xi	\in	L^{4/3}(\overline\Omega;L^{4/3}(0,T;(W^{-1,4/3}(D))^d))$	and	
	$y_\infty\in L^2(\Omega\times	(0,T);V)\cap	L^4(\Omega\times(0,T);X)\cap		L^2_w(\Omega;L^\infty(0,T;H))$,		a	$(\overline{\mathcal{F}}^{\infty}_t)_{t\in [0,T]}$-predictable process	such	that	
	the	following	convergences	hold	(up	to	subsequence	denoted	by	the	same	way),	as	$k\to	\infty$:
	\begin{align}
	\hspace*{-4cm}	&\overline{y_k}  \text{ converges strongly to  } y_\infty \text{ in } L^2(\overline\Omega;L^2(0,T;H)) 
	\label{cv2}\\ 
	&\overline{y_k}  \text{ converges weakly to  } y_\infty \text{ in } L^2(\overline\Omega;L^2(0,T;V))\cap	L^4(\overline\Omega\times(0,T);X);\label{cv1}\\
	&S(\overline{y_k})\text{ converges weakly to  }	\xi\text{ in }	L^{4/3}(\overline\Omega;L^{4/3}(0,T;X^\prime));\label{cv-mono-oper}\\
	&\overline{y_0^k} \text{ converges  to  } \bar y_0 \text{ in }  L^2(\overline\Omega;H)  \label{initial-cv}; \\
	&\overline{y_k}  \text{ converges weakly-* to  } y_\infty \text{ in } L^2_{w-*}(\overline\Omega;L^\infty(0,T;H)) \label{cv-*},
	\end{align}
	where	$L^2_{w-*}(\overline\Omega;L^\infty(0,T;H))$	denotes	the	space	$$\{ u:\overline\Omega\to	L^\infty(0,T;H)		\text{	is	 weakly-* measurable		and	}	\overline\E\Vert	u\Vert_{L^\infty(0,T;H)}^2<\infty\}.$$
\end{lemma}
\begin{proof}
	From Lemma \ref{skorohod-cv}, we know that
	$$ \overline{ y_k}  \text{ converges strongly to  } y_\infty \text{ in } L^2(0,T;H)\quad  \overline{P}\text{-a.s. in  } \overline\Omega. $$
	Then the	Vitali’s theorem yields  	\eqref{cv2}, since $(\overline{y_k})_k$	is	bounded	in	$L^4(\overline\Omega\times(0,T);X)$. 
	\\

	By the compactness of the closed balls  
	in the space $L^2(\overline\Omega;L^2(0,T;V))\cap	L^4(\overline\Omega\times(0,T);(W^{1,4}(D))^d)$	and	$L^2_{w-*}(\overline\Omega;L^\infty(0,T;H))$
	with respect to the weak 	and	weak-*	topologies,	respectively, there exists $$\Xi \in L^2(\overline\Omega;L^2(0,T;V))\cap	L^4(\overline\Omega\times(0,T);X)\cap	L^2_{w-*}(\overline\Omega;L^\infty(0,T;H))	$$ such that $\overline{y_k} \rightharpoonup \Xi$	in	$L^2(\overline\Omega;L^2(0,T;V))\cap	L^4(\overline\Omega\times(0,T);X)$	and	$\overline{ y_k} \rightharpoonup_* \Xi$	in	$L^2_{w-*}(\overline\Omega;L^\infty(0,T;H))$, then  the uniqueness of the limit gives $\Xi=y_\infty$. 
	A similar argument yields 	the	existence	of	$\xi	\in	L^{4/3}(\overline\Omega;L^{4/3}(0,T;X^\prime))$	such	that	\eqref{cv-mono-oper}	holds.
	\\	
	
	Concerning	the	$(\overline{\mathcal{F}}^{\infty}_t)_{t\in [0,T]}$-predictability	of		$y_\infty$,	it is 	clear	that	$y_\infty$	is		$(\overline{\mathcal{F}}^{\infty}_t)_{t\in [0,T]}$-adpated.	Since	$y_\infty\in	C([0,T];H_{weak})$	$\overline{P}$-a.s.,	see	\eqref{space-tightness-sol}.	Then,	the		$(\overline{\mathcal{F}}^{\infty}_t)_{t\in [0,T]}$-predictability	of	$y_\infty$	follows.\\
	
	Thanks	to	the	equality	of	laws,	one	has	$$\displaystyle\sup_{k\in\mathbb{N}}\overline{\E}\Vert	\overline{y_0^k}\Vert_H^r=\sup_{k\in\mathbb{N}}\E\Vert	P_ky_0\Vert_H^r\leq\E\Vert	y_0\Vert_H^r.$$	On	the	other	hand,
	$\overline{y_0^k}$	converges	to	$\bar{y}_0$	in	$H$	$\overline P$-a.s. in  $ \overline \Omega$,	consequently,
	Vitali's	theorem	ensures	that	$\overline{y_0^k}$  converges  to $\bar y_0$	in $L^p(\overline\Omega;H),	1\leq	p<r$.	Moreover,	we	have
	$\mathcal{L}(\bar{y}_0)=\mathcal{L}(y_0).$	
\end{proof}
\begin{lemma}\label{Lemma-cv-*}
	For	any	$t\in[0,T]$,
	the	following	convergences	hold.
	\begin{align}
	\int_0^t	P_kG(\cdot,\overline{y_k})d\overline{	\mathcal{W}}_k&\to	\int_0^t	G(\cdot,y_\infty)d\mathcal{W}_\infty	\text{ in } L^2(\overline\Omega;L^2(0,T;(L^2(D))^d));\label{cv-stoch-integ}\\
	\quad	\overline{	y_k}(t)	&\text{ converges weakly	 to  }y_\infty(t)\text{ in } L^2(\overline\Omega;H)	\text{	and	}	\bar y_0= y_\infty(0).\label{cont-time}
	\end{align}
\end{lemma}
\begin{proof}
	First,	note	that
	\begin{align*}
	&\overline\E\int_0^T\Vert	P_kG(s,\overline{y_k}(s))-G(s,y_\infty(s))\Vert_{L_2(\mathbb{H},(L^2(D))^d)}^2ds\\
	&\leq	2\overline\E\int_0^T\Vert	G(s,\overline{y_k}(s))-G(s,y_\infty(s))\Vert_{L_2(\mathbb{H},(L^2(D))^d)}^2ds+2\overline\E\int_0^T\Vert	(P_k-Id)G(s,y_\infty(s))\Vert_{L_2(\mathbb{H},(L^2(D))^d)}^2ds\\
	&\leq	2L\overline\E	\int_0^T\Vert	\overline{y_k}-y_\infty	\Vert_2^2ds+2L\Vert	P_k-Id\Vert_{L(H)}^2\overline\E	\int_0^T\Vert	y_\infty\Vert_2^2ds	\to	0	\text{	as	}	n\to\infty,
	\end{align*}
	by	using	\eqref{cv2}	and	the	properties	of	the	projection	operator	$P_n$.	From	Lemma	\ref{Wienr-limit},	we	have
	$\overline{\mathcal{W}}_k$	converges	to	$\mathcal{W}_\infty$	in	$L^2(\overline{\Omega},\mathcal{C}([0,T], H_0))$.	In	addition,	$G(\cdot,y_\infty)	\in	L^2(0,T;L_2(\mathbb{H},(L^2(D))^d))$	is	$\overline{\mathcal{F}}^\infty_t$-predictable,	since	$y_\infty$	is	$\overline{\mathcal{F}}^\infty_t$-predictable	and		$G$	satisfies	\eqref{noise1}.
	Now,	we	are	in	position	to	use	\cite[Lemma	2.1]{Debussche11}	and	deduce
for	any	$t\in[0,T]$	\begin{align*}
	\int_0^t	P_kG(\cdot,\overline{y_k})d\overline{	\mathcal{W}}_k\to	\int_0^t	G(\cdot,y_\infty)d\mathcal{W}_\infty	\text{ in	probability	in  } L^2(0,T;(L^2(D))^d)).
	\end{align*}
	To	obtain	the	first	claim	of	Lemma	\ref{Lemma-cv-*},	note	that	for	any	$t\in[0,T]$
	\begin{align*}
	\overline{\E}\vert	\int_0^t	P_kG(\cdot,\overline{y_k})d\overline{	\mathcal{W}}_k\vert^4\leq	C\overline{\E}\big[\sum_{\k \ge 1}\int_0^{T}\Vert\sigma_\k(\cdot,\overline{y_k})\Vert_{2}^2ds\big]^{2}\leq	CL\overline{\E}\big[\int_0^{T}\Vert\overline{y_k}\Vert_{2}^2ds\big]^{2}\leq	CLT\overline{\E}\big[\int_0^{T}\Vert\overline{y_k}\Vert_{2}^4ds\big]\leq	K,
	\end{align*}		
	since	$(\overline{y_k})_k$	is	bounded	by	$K$	in	$	L^4(\Omega\times(0,T);X)$.	Hence,	$(\int_0^\cdot	P_kG(\cdot,\overline{y_k})d\overline{	\mathcal{W}}_k)_k$	is	uniformly	integrable	in	$L^p(\overline{	\Omega}),1\leq	p<4$	and	 Vitali's  theorem	implies	\eqref{cv-stoch-integ}.\\
	
	From	Lemma	\ref{skorohod-cv}	and	Lemma	\ref{approx-martingale}$_{(2)}$,	it follows that	$\overline{	y_k}\to	y_\infty$	in	$C([0,T];U^\prime)$	$\overline{P}$-a.s.	and	$(\overline{y_k})_k$	is	bounded	by	$K$	in	$L^2(\overline\Omega;C([0,T];H))$.	Since	$H\hookrightarrow	U^\prime$,	there	exist	$C>0$	independent	of	$k$	such	that
	\begin{align*}
	\sup_{k\in	\mathbb{N}}\overline{\E}\sup_{s\in [0,T]}\Vert	\overline{y_k}(s)\Vert_{U^\prime}^2\leq	C\sup_{k\in	\mathbb{N}}\overline{\E}\sup_{s\in [0,T]}\Vert	\overline{y_k}(s)\Vert_{H}^2\leq	CK^2.
	\end{align*}
	Hence,	Vitali's	theorem	esnures	that	$\overline{	y_k}\to	y_\infty$	in	$L^q(\overline{\Omega};C([0,T];U^\prime)$		for	any	$1\leq	q	<2$
	and	for	any	$t\in	[0,T]$:	$	\overline{y_k}(t)	\rightharpoonup	y_\infty(t)	\text{	in	}		L^q(\overline{\Omega};U^\prime).$	Recall	that
	$(\overline{y_k})_k$	is	bounded		in	$L^2(\overline\Omega;C([0,T];H))$	to	obtain	\eqref{cont-time}	and	$\bar y_0= y_\infty(0)$.
	
\end{proof}
\subsubsection{Step	2:	Passage	to	the	limit	and	identification	of	limits}	From		\eqref{eqn-martingale-appr},	we	have	\begin{align}\label{Pass-to-limit}
(\overline{	y_k}(t),e_i)=(\overline{y_0^k},e_i)&+	 \displaystyle\int_0^t\langle	F+\nu	\Delta \overline{	y_k}-\overline{	y_k}\cdot \nabla	\overline{	y_k}+\alpha	\text{div}(A(\overline{	y_k})^2)+\beta		\text{div}(|A(\overline{	y_k})|^2A(\overline{	y_k})),e_i\rangle	ds\nonumber\\
&\quad+\int_0^{t} (G(\cdot,\overline{y_k})\,d\overline{	\mathcal{W}}_k,e_i) \\
=(\overline{y_0^k},e_i)&+	 \displaystyle\int_0^t\langle	F-\overline{	y_k}\cdot \nabla	\overline{	y_k}-S(\overline{y_k}),e_i\rangle	ds+\int_0^{t} (G(\cdot,\overline{y_k})\,d\overline{	\mathcal{W}}_k,e_i);\quad	i=1,\cdots,k.\nonumber
\end{align}
Thus,	there	exist		a	$H$-valued	square-integrable		$(\overline{\mathcal{F}}^{\infty}_t)_{t\in [0,T]}$-predictable	process			denoted	by	$y_\infty$	and	a	predictable	process	$\xi$	belongs	to	$	L^{4/3}(\overline\Omega;L^{4/3}(0,T;X^\prime))$	such	that	
\begin{align*}
y_\infty\in	L^2(\overline\Omega\times	(0,T);V)\cap	L^4(\overline\Omega\times(0,T);X)\cap	L^2_{w-*}(\overline\Omega;L^\infty(0,T;H)).
\end{align*}
\begin{enumerate}
	\item[(i)]	By	using	Lemma	\ref{lemma-cv-limit}	and	Lemma	\ref{Lemma-cv-*},	passing	to	the	limit	in	\eqref{Pass-to-limit}	as	$k\to\infty$,	we	obtain	for	almost	every	$(\overline{\omega},t)\in\overline{\Omega}\times[0,T]$	
	\begin{align}\label{eqn-limit1}
	(y_\infty(t),e_i)=(y_\infty(0),e_i)&+	 \displaystyle\int_0^t\langle	F-y_\infty\cdot \nabla	y_\infty-\xi,e_i\rangle	ds+ (\int_0^t		G(\cdot,y_\infty)d\mathcal{W}_\infty,e_i);\quad	\forall	i\in\mathbb{N} 
	\end{align}
	with	respect	to	new	stochastic	basis	$(\overline{\Omega},\overline{\mathcal{F}},\overline{P};(\overline{\mathcal{F}}^{\infty}_t)_{t\in [0,T]})$,	associated	with	$\mathcal{W}_\infty$.
	\item[(ii)]	Since	$y_\infty$	is	$(\overline{\mathcal{F}}^{\infty}_t)_{t\in [0,T]}$-predictable,	belongs	to	$	L^2(\overline\Omega\times	(0,T);V)$	and		$\mathcal{W}_\infty=(\mathcal{W}_\infty(t))_{t\in [0,T]}$ is a $H_0$-valued,	$Q$-Wiener	process,	we	obtain	that		$(\int_0^t		G(\cdot,y_\infty)d\mathcal{W}_\infty)$	is	 $(L^2(D))^d$ -valued continuous square integrable martingale	and	\eqref{eqn-limit1}	holds	for	any	$t\in[0,T]$.	Moreover,	by	using	\eqref{eqn-limit1}	one	has	$\overline{P}$-a.s.	$\overline{\omega}\in	\overline{\Omega}:$ $y_\infty(\overline{\omega},\cdot
	)\in	C([0,T];X^\prime)\cap L^\infty(0,T;H)$,	which		yields $y_\infty(\overline{\omega})\in C([0,T];H_w)$.
	
	\item[(iii)]	
	Since	$U$	is	separable	Hilbert	space,	\eqref{eqn-limit1}	holds	for	any	$v\in	U$.	By	taking	into	account	Lemma	\ref{lemma-cv-limit}	and		the	regularity	of	$y_\infty$,	we	obtain	$$F-y_\infty\cdot \nabla	y_\infty-\xi	\in	L^{4/3}(\overline\Omega;L^{4/3}(0,T;X^\prime))$$
	and,	by	density	argument	
	\begin{align}
	(y_\infty(t),v)=(y_\infty(0),v)&+	 \displaystyle\int_0^t\langle	F-y_\infty\cdot \nabla	y_\infty-\xi,v\rangle	ds+ (\int_0^t		G(\cdot,y_\infty)d\mathcal{W}_\infty,v);\quad	\forall	v\in	X. 
	\end{align}Thus,  we have back $y_\infty(\overline{\omega})\in C([0,T];H)$, according to \cite[Thm 4.2.5]{Liu-Rock}.
\end{enumerate}

Since	the	last	equality	holds	in	$X^\prime$-sense,	we	can		apply	It\^o's	formula	for	$\Vert	\cdot\Vert_2^2$	to	get	
\begin{align}\label{lim-mon}
\Vert	y_\infty(t)\Vert_2^2&=\Vert	y_\infty(0)\Vert_2^2+	 2\displaystyle\int_0^t\langle	F	-y_\infty\cdot \nabla	y_\infty-\xi,y_\infty\rangle	ds+ 2(\int_0^t		G(\cdot,y_\infty)d\mathcal{W}_\infty,y_\infty)+\sum_{\k\geq 1}\int_0^t\Vert\sigma_{\k}(\cdot,y_\infty)\Vert_{2}^2ds\notag\\
&=\Vert	y_\infty(0)\Vert_2^2-	 2\displaystyle\int_0^t\langle	\xi-F,y_\infty\rangle	ds+ 2(\int_0^t		G(\cdot,y_\infty)d\mathcal{W}_\infty,y_\infty)+\sum_{\k\geq 1}\int_0^t\Vert\sigma_{\k}(\cdot,y_\infty)\Vert_{2}^2ds.
\end{align}
By	using		\eqref{eqn-martingale-appr}	and	similair	similair arguments as the proof of Theorem	\eqref{Thm-app},	we	obtain
\begin{align}\label{app-mon}
\Vert	\overline{	y_k}(t)\Vert_2^2&=\Vert\overline{y_0^k}\Vert_2^2	-	 2\displaystyle\int_0^t\langle	S(\overline{	y_k})-F,\overline{	y_k}\rangle	ds+ 2(\int_0^t		G(\cdot,\overline{	y_k})d\overline{\mathcal{W}}_k,\overline{	y_k})+\sum_{\k\geq 1}\int_0^t\Vert	P_k\sigma_{\k}(\cdot,\overline{	y_k})\Vert_{2}^2ds.
\end{align}
Therefore,	after subtraction	\eqref{lim-mon}	from	\eqref{app-mon}
\begin{align*}
&\Vert	\overline{	y_k}(t)\Vert_2^2-\Vert	y_\infty(t)\Vert_2^2=\Vert\overline{y_0^k}\Vert_2^2-\Vert	y_\infty(0)\Vert_2^2+2\displaystyle\int_0^t[\langle	\xi,y_\infty\rangle-\langle	S(\overline{	y_k}),\overline{	y_k}\rangle]	ds+2\int_0^t\langle	F,\overline{	y_k}-y_\infty\rangle	ds\\
&+2(\int_0^t		G(\cdot,\overline{	y_k})d\overline{	\mathcal{W}}_k,\overline{	y_k})-2(\int_0^t		G(\cdot,y_\infty)d\mathcal{W}_\infty,y_\infty)+\sum_{\k\geq 1}\int_0^t[\Vert	P_k\sigma_{\k}(\cdot,\overline{	y_k})\Vert_{2}^2-\Vert\sigma_{\k}(\cdot,y_\infty)\Vert_{2}^2]ds
\end{align*}
Since	$\overline{	y_k}(t)$	converges	weakly	to		$y_\infty(t)$	in	$L^2(\overline{ \Omega},H)$	for	any	$t\in[0,T]$,	see	\eqref{cont-time},	we	get
$$\liminf_{k}[\overline{\E}\Vert	\overline{	y_k}(t)\Vert_2^2-\overline{\E}\Vert	y_\infty(t)\Vert_2^2]\geq	0,	\forall	t\in	[0,T].$$
Set	$t=T$,	take	the	expectation		and	pass	to	the	limit	as	$k\to\infty$	to	obtain
\begin{align}\label{eqn-identif}
0\leq	\liminf_{k}\overline\E\displaystyle\int_0^T[\langle	\xi,y_\infty\rangle-\langle	S(\overline{	y_k}),\overline{	y_k}\rangle]	ds,
\end{align}
where	we	used	Lemma	\ref{lemma-cv-limit}	to	obtain	the	last	inequality.
Now,	let	$v\in		L^4(\overline\Omega\times(0,T);X)$		and	note	that	
\begin{align*}
\overline\E\displaystyle\int_0^T\langle	\xi-S(v),y_\infty-v\rangle	ds&=\overbrace{\overline\E\displaystyle\int_0^T\langle	S(\overline{y_k})-S(v),\overline{y_k}-v\rangle	ds}^{\geq	0,\text{	thanks	to	Corollary	\ref{Monotone-S-ope}	}}+\overbrace{\overline\E\displaystyle\int_0^T\langle	S(\overline{y_k})-\xi,v\rangle	ds}^{\to	0\text{,	thanks	to	\eqref{cv-mono-oper}	}}\\&+\overbrace{\overline\E\displaystyle\int_0^T\langle	S(v),\overline{y_k}-y_\infty\rangle	ds}^{\to	0\text{,	thanks	to	\eqref{cv1}	}}+\overbrace{\overline\E\displaystyle\int_0^T[\langle	\xi,y_\infty\rangle-\langle	S(\overline{	y_k}),\overline{	y_k}\rangle]	ds}^{\geq	0\text{,	thanks	to	\eqref{eqn-identif}	}}.\end{align*}
Therefore,	we	get	
\begin{align*}
\overline\E\displaystyle\int_0^T\langle	\xi-S(v),y_\infty-v\rangle	ds	\geq	0,\quad	\forall	v	\in		L^4(\overline\Omega\times(0,T);X).
\end{align*}
Let	$\lambda\in	\mathbb{R}$,
by	using	a	Minty's	trick,	namely	take	$v=y_\infty+\lambda\phi$	to	obtain
\begin{align*}
\overline\E\displaystyle\int_0^T\langle	\xi-S(y_\infty+\lambda\phi),-\lambda\phi\rangle	ds	\geq	0,\quad	\forall	\phi	\in		L^4(\overline\Omega\times(0,T);X).
\end{align*}
Thus,	by	letting	$\lambda\to0$,	we	are	able	to	deduce	after	routine	steps	(see	\textit{e.g.}	\cite[Lemma	2.13]{Roubicek})	that	$S(y_\infty)=\xi$.	Finally,	there	exist	a		stochastic	basis	$(\overline{\Omega},\overline{\mathcal{F}},\overline{P};(\overline{\mathcal{F}}^{\infty}_t)_{t\in [0,T]})$	and	Wiener	process	$\mathcal{W}_\infty$	such	that:
\begin{itemize}
	\item	there	exist	$y_\infty$,	$(\overline{\mathcal{F}}^{\infty}_t)_{t\in [0,T]}$-adapted,	with		continuous	paths	in	$H$.
	\item	$y_\infty\in	L^2(\overline\Omega\times	(0,T);V)\cap	L^4(\overline\Omega\times(0,T);X)\cap	L^2_{w-*}(\overline\Omega;L^\infty(0,T;H))$.
	\item	$\overline	P$-a.s.	for	any	$t\in[0,T]$	
	\begin{align*}
	(y_\infty(t),v)=(y_\infty(0),v)&+	 \displaystyle\int_0^t\langle	F-y_\infty\cdot \nabla	y_\infty+\nu\Delta	y_\infty+\alpha\text{ div}(A(y_\infty)^2)+\beta\text{ div}(\vert	A(y_\infty)\vert^2A(y_\infty)),v\rangle	ds\\
	&+ \int_0^t(		G(\cdot,y_\infty),v)d\mathcal{W}_\infty;\quad	\forall	v\in	X.
	\end{align*}\end{itemize}
	\subsection{A	stability	result	and	pathwise	uniqueness}
\begin{lemma}\label{Lemma-V-stability} Assume that $(\mathcal{W}(t))_{t\geq 0}$ is a $Q$-Wiener process in $H_0$ with respect to the stochastic basis $(\Omega,\mathcal{F},P; (\mathcal{F}_t)_{t\geq 0})$ and $y_1,y_2$ are two  solutions  in the sense of Definition \ref{strongsol-def} to \eqref{I} with respect to the initial conditions $y_0^1,y_0^2$, on $(\Omega,\mathcal{F},P; (\mathcal{F}_t)_{t\geq 0})$. Then
\begin{align*}
&\E	\sup_{r\in [0,t]}g(r)\Vert (y_1-y_2)(r)\Vert_2^2\leq	2\E\Vert y_0^1-y_0^2\Vert_2^2e^{2C(L)t}	,	\quad	\forall	t\in[0,T],	\end{align*}
where	$g(t)=e^{-\frac{C_1^2}{\nu\epsilon_0}\int_0^t\Vert	\nabla	y_1(r)\Vert_{L^3}^2dr}$	and	$C_1>0$,	depending	only	on	$H^1_0(D)\hookrightarrow L^6(D)$.
\end{lemma}
\begin{proof}
	Let $y_1$ and $y_2$ be two solutions of \eqref{I}  associated	to   the initial data $y_0^1$	and	$y_0^2$, respectively.\\

	Set $y=y_1-y_2, y_0=y_0^1-y_0^2$	and	recall	that	$\text{ div	}	y=0$.  We have for any $t\in [0,T]$
	\begin{align*}
	y(t)-y_0&=-\int_0^t\nabla(\mathbf{\bar P}_1-\mathbf{\bar P}_2)ds+ \int_0^t\epsilon_0\nu\Delta y-\big[(y\cdot \nabla)y_1+(y_2\cdot \nabla)y\big]ds\\&+\beta\epsilon_0\int_0^t\text{div}\Big[\vert	A(y_1)\vert^2A(y_1)-\vert	A(y_2)\vert^2A(y_2)\Big]ds-\int_0^t[T(y_1)-T(y_2)]ds\\&\hspace*{4cm}+ \int_0^t[G(\cdot,y_1)-G(\cdot,y_2)]d\mathcal{W}. 
	\end{align*}
		By applying  It\^o formula,	see	\textit{e.g.}	\cite[Thm.	4.2]{Pardoux},	with	$F(y)=\Vert	y\Vert_{2}^2$, one gets	for	any	$t\in	[0,T]$
		\begin{align*}
				&\Vert y(t)\Vert_2^2-\Vert y_0\Vert_2^2=-2\epsilon_0\nu\int_0^t\Vert \nabla y\Vert_2^2ds-2\int_0^tb(y,y_1,y) ds\\&+2\beta\epsilon_0\int_0^t\langle	\text{div}\Big[\vert	A(y_1)\vert^2A(y_1)-\vert	A(y_2)\vert^2A(y_2)\Big],y_1-y_2\rangle	ds -2\int_0^t
			\langle T(y_1)-T(y_2),y_1-y_2\rangle ds\vspace{2mm}\\&\quad+2\int_0^t(G(\cdot,y_1)-G(\cdot,y_2), y_1-y_2)d\mathcal{W}+\sum_{\k\ge 1}\int_0^t\Vert  \sigma_\k(\cdot,y_1)-\sigma_\k(\cdot,y_2)\Vert_2^2ds.
		\end{align*}
		Now,	denote	by	$g$	the	following	function	$g(t)=e^{-C\int_0^t\Vert	\nabla	y_1(r)\Vert_{L^3}^2dr}$,	$C>0$ (to	be	chosen),	we	recall	the	following	"stochastic"	integration	by	parts	formula	
\begin{align*}
					g(t)\Vert y(t)\Vert_2^2-\Vert y_0\Vert_2^2=\int_0^tg(s)d\Vert y(s)\Vert_2^2+\int_0^tg^\prime(s)\Vert y(s)\Vert_2^2ds.
	\end{align*}
Therefore,	we	get
		\begin{align*}
	&g(t)\Vert y(t)\Vert_2^2-\Vert y_0\Vert_2^2+C\int_0^tg(s)\Vert	\nabla	y_1\Vert_{L^3}^2\Vert y\Vert_2^2ds+2\epsilon_0\nu\int_0^tg(s)\Vert \nabla y\Vert_2^2ds\\&=-2\int_0^tg(s)b(y,y_1,y) ds+2\beta\epsilon_0\int_0^tg(s)\overbrace{\langle	\text{div}\Big[\vert	A(y_1)\vert^2A(y_1)-\vert	A(y_2)\vert^2A(y_2)\Big],y_1-y_2\rangle}^{\leq	0,\text{	see	\cite[Lemma	2.4]{Paicu}	}}	ds \vspace{2mm}\\&\quad-2\int_0^t
g(s)	\overbrace{\langle T(y_1)-T(y_2),y_1-y_2\rangle}^{\geq	0,\text{	thanks	to	Lemma	\ref{Lemma-monotone}	}} ds+2\int_0^tg(s)(G(\cdot,y_1)-G(\cdot,y_2), y_1-y_2)d\mathcal{W}\\&\qquad+\sum_{\k\ge 1}\int_0^tg(s)\Vert  \sigma_\k(\cdot,y_1)-\sigma_\k(\cdot,y_2)\Vert_2^2ds.	\end{align*}
Thus
	\begin{align}
& g(t)\Vert y(t)\Vert_2^2-\Vert y_0\Vert_2^2+C\int_0^tg(s)\Vert	\nabla	y_1\Vert_{L^3}^2\Vert y\Vert_2^2ds+2\epsilon_0\nu\int_0^tg(s)\Vert \nabla y\Vert_2^2ds\\&\leq	2\int_0^tg(s)\vert	b(y,y_1,y)\vert	 ds+\sum_{\k\ge 1}\int_0^tg(s)\Vert  \sigma_\k(\cdot,y_1)-\sigma_\k(\cdot,y_2)\Vert_2^2ds.\notag
\\
&\qquad+2\int_0^t(G(\cdot,y_1)-G(\cdot,y_2), y_1-y_2)d\mathcal{W}\notag	\end{align}
Since	$y_1\in	L^4(\Omega\times(0,T);(W^{1,4}_0(D))^d)$  and 
	 $H^1_0(D)\hookrightarrow L^6(D)$,	then there	exists	$C_1>0$ such	that
	 	\begin{align*}
	 \vert	b(y,y_1,y)\vert 
	&=\left\vert \int_D(y\cdot \nabla)y_1\cdot y dx\right\vert \leq \Vert y\Vert_6\Vert \nabla y_1\Vert_3\Vert	y\Vert_{2}\leq \dfrac{\nu\epsilon_0}{2}\Vert	\nabla	y\Vert_2^2+\dfrac{C_1^2}{2\nu\epsilon_0}\Vert	\nabla y_1\Vert_3^2\Vert		y\Vert_2^2. 
	\end{align*}
	By	using	\eqref{noise1},	one	has
	\begin{align*}
		\sum_{\k\ge 1}\int_0^tg(s)\Vert  \sigma_\k(\cdot,y_1)-\sigma_\k(\cdot,y_2)\Vert_2^2ds\leq	L\int_0^tg(s)\Vert  y(s)\Vert_2^2ds.
	\end{align*}
	By	gathering	the	above	estimate,	we	obtain
	\begin{align*}
		g(t)\Vert y\Vert_2^2(t)&-\Vert y_0\Vert_2^2+(C-\dfrac{C_1^2}{\nu\epsilon_0})\int_0^tg(s)\Vert	\nabla	y_1\Vert_{L^3}^2\Vert y\Vert_2^2ds+\epsilon_0\nu\int_0^tg(s)\Vert \nabla y\Vert_2^2ds\\&\leq	L\int_0^tg(s)\Vert  y(s)\Vert_2^2ds
		+2\int_0^t(G(\cdot,y_1)-G(\cdot,y_2), y_1-y_2)d\mathcal{W}.	\end{align*}
	By	choosing	$C=\dfrac{C_1^2}{\nu\epsilon_0}$	and	taking	the	expectation,	we	infer	that		for	any	$t\in[0,T]$	
		\begin{align*}
	&\E	\sup_{r\in [0,t]}g(r)\Vert y(r)\Vert_2^2\leq	\E\Vert y_0\Vert_2^2+	L\E\int_0^tg(s)\Vert  y(s)\Vert_2^2ds+2\E\sup_{s\in [0,t]}&\vert\int_0^sg(s)(G(\cdot,y_1)-G(\cdot,y_2), y_1-y_2)d\mathcal{W}\vert.	\end{align*}
	Note	that	$g\in	L^\infty(\Omega_T)$,	thus
	$(\displaystyle\int_0^{t}g(s)(G(\cdot,y_1)-G(\cdot,y_2), y_1-y_2)d\mathcal{W})$	is	$(\mathcal{F}_t)_{t\in [0,T]}$-martingale.
	Let	$t\in	]0,T]$,
	by using Burkholder–Davis–Gundy  and Young inequalities, 	 we deduce	the	existence	of	$C_B>0$	such	that
	\begin{align*}
	2\E\sup_{s\in [0,t]}&\vert\int_0^sg(s)(G(\cdot,y_1)-G(\cdot,y_2), y_1-y_2)d\mathcal{W}\vert 
	\\&\leq C_B\E\big[\sum_{\k \ge 1}\int_0^{t}(g(s))^2\Vert\sigma_\k(\cdot,y_1)-\sigma_\k(\cdot,y_2)\Vert_{2}^2\Vert y_1-y_2\Vert_2^2ds\big]^{1/2}\\
	&\leq \dfrac{1}{2} \E \sup_{s\in [0,t]} g(s)\Vert y\Vert_2^2+C(L) \E\int_0^{t} g(s)\Vert y \Vert_2^2ds.
	\end{align*}
	where	$C(L):=	2L(1+2C_B^2)$.	Therefore
\begin{align*}		&\dfrac{1}{2}\E	\sup_{r\in [0,t]}g(r)\Vert y(r)\Vert_2^2\leq	\E\Vert y_0\Vert_2^2+	C(L)\E\int_0^tg(s)\Vert  y(s)\Vert_2^2ds.	\end{align*}

	Finally, Gronwall's inequality ensures Lemma \ref{Lemma-V-stability}.
	\end{proof}
\begin{cor}\label{cor-pathwise-uniqueness}
Let	$y_1,y_2$	be	two	solutions	to	\eqref{I}	defined	on		  $(\Omega,\mathcal{F},P; (\mathcal{F}_t)_{0\leq	t\leq T}$		with	the	same	$(\mathcal{W}(t))_{0\leq	t\leq T})$ such	that	$y_i(0)=y_0,i=1,2.$	Then
\begin{align*}
P\big[	y_1(t)=y_2(t)\big]=1	\text{	for	every	}	t\in[0,T].
\end{align*}
\end{cor}
\begin{proof}
Let	$t\in	[0,T]$,	thanks	to	Lemma	\ref{Lemma-V-stability}	we	have
$
\E	g(t)\Vert (y_1-y_2)(t)\Vert_2^2	=	0.	$
	Define	the	following	sequence	of	stopping	times
\begin{align}
\tau_N=\inf\{t:	0\leq	t\leq	T;	\quad	\int_0^t\Vert	\nabla	y_1(r)\Vert_{L^3}^2dr>N	\},\quad	N\in	\mathbb{N}^*.
\end{align}
It	follows		that	$\E	\Vert (y_1-y_2)(t\wedge	\tau_N)\Vert_2^2	=	0$.	On	the	other	hand,	note	that
	\begin{align*}
	NP(\tau_N <T)\leq	\E( 1_{\{ \tau_N<T\}} \int_0^t\Vert	\nabla	y_1(r)\Vert_{L^3}^2dr) \leq 	C\E  \int_0^T\Vert	\nabla	y_1(r)\Vert_{W^{1,4}_0}^2dr \leq 	\mathbf{C}.
	\end{align*}
	Therefore	 $\tau_N \to T$ in probability, as $N\to \infty$.	On	the	other	hand,	
	 $\{\tau_N\}_N$ is  an 
	increasing sequence, then the monotone convergence (Beppo Levi's) theorem	 allows to pass to the limit in		$\E	\Vert (y_1-y_2)(t\wedge	\tau_N)\Vert_2^2	=	0$,	as	$N\to	\infty$	and deduce that	
	$\E	\Vert (y_1-y_2)(t)\Vert_2^2	=	0$,	which	completes	the	proof.
\end{proof}
\section{Existence	of	an	invariant	measure	and	ergodicity}\label{Section-invariant-measure}
In	this	section,	we	are	interested	to	prove	the	existence	of	an	ergodic		invariant	measure	under	some	natural	assumptions.	For	that,	let	us	precise	the	assumptions	on	the	data	to	show	the	desired	result.\\

	Let	$\theta_\k:  \R^d\mapsto \R^d, \; \k \in\mathbb{N},$  	be	 
a family of Carath\'eodory functions
such	that
$\theta_\k(0)=0$,  and   there exists $L > 0$  such that   
\begin{align}
\label{noise1-2}
\quad &\sum_{\k\ge 1}\big| \theta_\k(\lambda)-\theta_\k(\mu)\big|^2  \leq L  |\lambda-\mu|^2;	\quad	\forall	\lambda,\mu \in \R^d.
\end{align}

For any  $H$-valued	predictable	process	$y$, we	define	a	Hilbert-Schmidt	operator	$\mathbb{G}$	as	follows:
\begin{align}\label{noise-invariant-sec}
\mathbb{G}(y): \mathbb{H}\goto (L^2(D))^d, \qquad \mathbb{G}(y)f_\k= \{ x \mapsto \theta_\k\big(y(x)\big)\},
\quad \k \ge 1.
\end{align}

In	this	part,	we	assume		that:
\begin{itemize}
	\item[$\overline{\mathcal{H}}_1:$] 	$y_0\in	H$,	$F\in	X^\prime$	and	the	operator		$\mathbb{G}$	satisfies	\eqref{noise1-2}.
	
\end{itemize}

We	will	prove	the	existence	of	an	ergodic	invariant	measure	for	 the system \eqref{I}, 
which reads	
\begin{align}\label{I-I}
\begin{cases}
dy=\big(F-\nabla \textbf{P}+\nu \Delta y-(y\cdot \nabla)y+\alpha\text{div}(A^2) +\beta \text{div}(|A|^2A)\big)dt+ \mathbb{G}(y)d\mathcal{W} \quad &\text{in } D \times (0,T)\times\Omega,\\
\text{div}(y)=0 \quad &\text{in } D \times (0,T)\times\Omega,\\
y=0 &\text{on } \partial D \times (0,T)\times\Omega,\\
y(x,0)=y_0(x) \quad &\text{in } D,
\end{cases}
\end{align}
As	a	consequence	of	Theorem	\ref{exis-thm-strong},	we	state the next theorem.
\begin{theorem}
	Assume that	$\overline{\mathcal{H}}_1$	holds. Then, there exists a strong solution  to \eqref{I-I} in the sense of Definition \ref{strongsol-def}.
\end{theorem}
	Let	$y(t;y_0),	t\geq	0$	be	the	unique	strong	solution	to	\eqref{I-I}.
For any bounded Borel function	$\varphi\in	\mathcal{B}_b(H)$	and	$t\geq	0$,	we	define
	\begin{align}\label{transition-prob}
	(P_t\varphi)(y_0)=\E[\varphi(y(t;y_0))],	\quad	y_0\in	H.
	\end{align}
	Notice	that	$(P_t)_{t\geq	0}$	is	a	stochastically	continuous	semigroup	on	the	Banach	space	$\mathcal{C}_b(H)$\footnote{$\mathcal{C}_b(H)$	denotes	the	set	of	real	valued	bounded	continuous	functions	on	$H$.}.	In	other	words,	
	$$	\forall	\varphi\in\mathcal{C}_b(H),\quad	y_0\in	H:	\quad	\lim_{t\to	0}P_t\varphi(y_0)=y_0.$$
	
By		similair	arguments to	the	one	used	in		\cite[Thm.	9.8	and	Cor.	9.9]{Daprato},	we	are	able	to	establish the result:
\begin{prop}\label{prop1}
	The	familly	$y(t,y_0),t\geq	0{\color{blue},}$	is	Markov	i.e.
	$$	\E[\varphi(y(t+s;\eta))\vert	\mathcal{F}_t]=(P_s\varphi)(y(t;\eta))	\quad	\forall	\varphi\in	\mathcal{C}_b(H),	\forall	\eta\in	H,	\forall	t,s>0.$$
	and	$P_{t+s}=P_tP_s$	for	$t,s\geq	0$.	
\end{prop}
Now,	let	us	prove	the	following	results		from	which		the	proof	of	Theorem	\ref{THm-invariant-measure-1}	follows.
\begin{prop}\label{prop-inva-1}
	The	semigroup	$(P_t)_t$	is	bounded	and	  Feller,	that	is,	 if $(y_0^n)_n	\subset	H$		converges	strongly	to $y_0$ 	in $H$ then
	$\displaystyle\lim_{n\to	\infty}P_t\phi(y_0^n)=	P_t\phi(y_0)$.
\end{prop}
\begin{proof}
		Let		$\phi\in	\mathcal{C}_b(H)$	and		$0<t\leq	T<\infty$,	$y_0\in	H$.		First,	it	is	clear	that	$P_t\phi:H\to		 \mathbb{R}$	is	bounded,	see	\eqref{transition-prob}.	
	Let	$\{y_0^n\}_n$	be	a	$H$-valued	sequence		such	that	$y_0^n$	converges	strongly	to	$y_0$	in	$H$	and
		denote	by	$y_n$	and	$y$	the		strong	solution	of		\eqref{I-I}		with	initial	data	$y_0^n$	and	$y_0$,	respectively.	Thanks	to	Lemma	\ref{Lemma-V-stability},	we	get
		\begin{align*}
		&\E	e^{-\frac{C_1^2}{\nu\epsilon_0}\int_0^t\Vert	\nabla	y(r)\Vert_{L^3}^2dr}\Vert y_n-y\Vert_2^2(t)\leq	2\Vert y_0^n-y_0\Vert_2^2e^{C(L)t}	.	\end{align*}
	Hence
$$
e^{-\frac{C_1^2}{\nu\epsilon_0}\int_0^t\Vert	\nabla	y(r)\Vert_{L^3}^2dr}\Vert y_n-y\Vert_2^2(t)\to 0 \quad \text{in } L^1(\Omega).
$$	
Then, there exists a subsequence,	denoted	by	the	same	way,	such	that
$$
\Vert y_n(t)-y(t)\Vert_2^2\to 0\quad  P-a.s, \quad \text{as } n\to\infty,
$$ 
 since the weight	is	positive	and does not depend on $n$.
					Thus,	$\phi(y_n(t))$	converges	to	$\phi(y(t))$	$P$-a.s.
								and	Lebesgue	dominated	convergence	theorem	ensures		
		\begin{align*}
		\phi(y_n(t))	\to	\phi(y(t))	\text{	in	}	L^1(\Omega)	\text{	for	any		}
	t\in[0,T].		\end{align*}
	Thus
	\begin{align*}
		\vert	P_t(\phi(y_0^n)-P_t(\phi(y_0))	\vert	\leq	\E\vert	\phi(y_n(t))-\phi(y(t))		\vert	\to	0	\text{	as		}
	n\to	\infty.	\end{align*}
		The uniqueness	of	the	solution	of	\eqref{I-I} ensures	the 	convergence	of	the	whole sequence	$(y_n(t))_n$,	 which completes the proof.
\end{proof}
\begin{prop}\label{prop-invar-2}
	Let	$y_0\in	H$	and	let	$y(t),t\geq0,$	be	the	unique	strong	solution	to	\eqref{I}	with	initial	data	$y_0$.	Then,	there	exists	$T_0\geq0$	such	that	for	every	$\epsilon>0$	there	exists	$R>0$	such	that
	\begin{align}\label{equ-invarinat}
\sup_{T\geq	T_0}\mu_T(H/\mathbb{B}_R)=	\sup_{T\geq	T_0}\dfrac{1}{T}\int_0^T(P_s^*\delta_{y_0})(H/\mathbb{B}_R)ds\leq	\epsilon,
	\end{align}
	where	$\mathbb{B}_R=\{v\in	V:	\Vert	v\Vert_V\leq	R\}.$		In	other	words,	the	set	$\{	\mu_n=\displaystyle\dfrac{1}{n}\int_0^{n}	P_s^*\delta_{y_0}	ds;	\hspace*{0.12cm}	n\in	\mathbb{N}^*\}$	is	tight.
\end{prop}
\begin{proof}
First,	we	recall	that	the	embedding	$V	\hookrightarrow	H$	is	compact	and	therefore	$\mathbb{B}_R$	is	 relatively	compact	in	$H$.
	Let	$0<t\leq	T$,	by
	applying	It\^o	formula	(see	\textit{	e.g.}	\cite[Theorem 4.2.5]{Liu-Rock})	with	$u\mapsto\Vert	u\Vert_{2}^2$	for	the	solution	of	\eqref{I},	we	get
	\begin{align}\label{Ito-y}
	\Vert	y(t)\Vert_2^2-\Vert	y_0\Vert_2^2&=2\nu\int_0^t\langle	\Delta	y,y\rangle	ds+2\int_0^t\langle	\alpha\text{div}(A^2),y\rangle	ds+2\int_0^t\langle	\beta \text{div}(|A|^2A),y\rangle	ds\notag\\
	&\quad+2\int_0^t\langle	F,y\rangle		ds+2\int_0^t(	G(\cdot,y),y)	d\mathcal{W}+	\int_0^t\sum_{\k\geq 1} \Vert	\sigma_{\k}(\cdot,y)\Vert_2^2ds.
	\end{align}
	Note	that	$(\displaystyle\int_0^t(	G(\cdot,y),y)	d\mathcal{W})_{0\leq	t\leq	T}$	is	$\mathcal{F}_t$-martingale	thus	$\E\displaystyle\int_0^t(	G(\cdot,y),y)	d\mathcal{W}=0$,	by	taking	the	expectation	and	using	\eqref{noise1},	we	get
	\begin{align}\label{est-invariant1}
	\E\Vert	y(t)\Vert_2^2-\Vert	y_0\Vert_2^2&\leq	2\E\int_0^t\langle	\nu\Delta	y+\alpha\text{div}(A^2)+\beta \text{div}(|A|^2A)+F,y\rangle	ds	+L\E	\int_0^t \Vert	y\Vert_2^2ds.
	\end{align}
	Integration	by	parts	ensures
	\begin{align}
	&2\int_0^t\langle	\nu\Delta	y+\alpha\text{div}(A^2)+\beta \text{div}(|A|^2A),y\rangle	ds\notag\\
	&=-\nu \int_0^t\Vert A( y)\Vert_{2}^2ds-2\alpha\int_0^t(A( y)^2,\nabla y)ds -\beta \int_0^t\int_D|A( y)|^4dxdt\notag\\
	&\leq		-2\nu \int_0^t\Vert \nabla y\Vert_{2}^2ds-\beta\int_0^t\int_D|A( y)|^4dxds+2\vert\alpha\vert\int_0^t\Vert	A(y)\Vert_4^2\Vert\nabla	y\Vert_2ds.
	\label{est-invar-2}\end{align}
		Since	$\epsilon_0=1-\sqrt{\dfrac{\alpha^2}{2\nu\beta}}\in]0,1[$,	we	get
	\begin{align}\label{est-invar-3}
	2\vert\alpha\vert\int_0^t\Vert	A(y)\Vert_4^2\Vert\nabla	y\Vert_2ds\leq	2\nu(1-\epsilon_0)\int_0^t\Vert \nabla y\Vert_{2}^2ds+\beta(1-\epsilon_0)\int_0^t\int_D|A( y)|^4dxds.
	\end{align}
	From	\eqref{est-invariant1},	\eqref{est-invar-2}	and	\eqref{est-invar-3},	we	infer	that	
	\begin{align*}
	\E\Vert	y(t)\Vert_2^2-\Vert	y_0\Vert_2^2&\leq	
	-2\nu\epsilon_0 \int_0^t\E\Vert \nabla y\Vert_{2}^2ds-\beta	\epsilon_0\int_0^t\E\int_D|A( y)|^4dxds	+L	\int_0^t \E\Vert	y\Vert_2^2ds+2\int_0^t\langle	F,y	\rangle	ds
	\end{align*}
	By	using	Poincar\'e's	inequality,	there	exists	$C_{P}>0$	such	that
	\begin{align}\label{poicare}
	\Vert	y(s)\Vert_2^2\leq	C_{P}\Vert \nabla y(s)\Vert_{2}^2	\text{	for	almost	all	}	s\in[0,t].
	\end{align}
	Recall	the	Korn	inequality	(see	\textit{e.g.}	\cite[Thm.	1.33]{Roubicek}),	there	exists	$C_K>0$	such	that	
	\begin{align}\label{Korn-inequlity}
	\Vert	y\Vert_{W^{1,4}_0}	\leq	C_K\Vert	\dfrac{\nabla	y+\nabla	y^T}{2}\Vert_4=	\dfrac{1}{2}C_K\Vert	A( y)\Vert_4
	\end{align}
	By	using	\eqref{poicare},	\eqref{Korn-inequlity}	and	Holder	inequality,	we	obtain
	\begin{align*}
	\Vert	y(s)\Vert_2^2&\leq	C_{P}\Vert \nabla y(s)\Vert_{2}^2\leq	C_{P}[\dfrac{1}{2\theta}\vert	D\vert+\dfrac{\theta}{2}\Vert	y(s)\Vert_{W^{1,4}_0}^4],\quad	\forall	\theta>0\\
	&\leq	C_{P}\dfrac{1}{2\theta}\vert	D\vert+C_{P}(C_K)^4\dfrac{\theta}{2^5}\Vert	A( y)\Vert_4^4,
	\end{align*}
	where	$\vert	D\vert$	denotes	the	Lebesgue	measure	of	the	domain	$D$.
	In	addition,	thanks	to	Poincaré	and	Korn	inequalities,	there	exists	$C_F>0$	such	that
	\begin{align}\label{exter-estim}
	2\int_0^t\langle	F,y	\rangle	ds	\leq	C\int_0^t\Vert	F\Vert_{X^\prime}\Vert	y\Vert_{W^{1,4}_0}ds	\leq	\dfrac{C_F}{(\beta	\epsilon_0)^{1/3}}t\Vert	F\Vert_{X^\prime}^{\frac{4}{3}}+\dfrac{\beta	\epsilon_0}{2}\int_0^t\Vert	A(y)\Vert_{4}^4ds,
	\end{align}
	where	we	used	Young	inequality	to	deduce	the	last	inequality.	Therefore
		\begin{align}\label{est-invariant2}
	\E\Vert	y(t)\Vert_2^2+2\nu\epsilon_0 \int_0^t\E\Vert \nabla y\Vert_{2}^2ds\leq	\Vert	y_0\Vert_2^2
	&+(LC_{P}(C_K)^4\dfrac{\theta}{2^5}-\dfrac{\beta	\epsilon_0}{2})\int_0^t\E\int_D|A( y)|^4dxds\notag\\	&+L	 C_{P}\dfrac{1}{2\theta}\vert	D\vert	t+\dfrac{C_F}{(\beta	\epsilon_0)^{1/3}}t\Vert	F\Vert_{X^\prime}^{\frac{4}{3}}.	\end{align}
	Thanks	to	Poincaré	inequality	\eqref{poicare},	one	has
$
		\Vert	y\Vert_V^2\leq	(C_P+1)\Vert	\nabla	y\Vert_2^2.
$
	Choose		$\theta=\dfrac{2^3\beta\epsilon_0}{LC_P(C_K)^4}$	and	using	\eqref{poicare},	we	obtain	
	\begin{align}\label{concentration-eqn}
	\E\Vert	y(t)\Vert_2^2&+\dfrac{2\nu\epsilon_0}{C_P+1} \int_0^t\E\Vert  y(s)\Vert_{V}^2ds+\dfrac{\beta	\epsilon_0}{4}\int_0^t\E\int_D|A( y)|^4dxds\notag\\&\quad¨\leq	\Vert	y_0\Vert_2^2
	+\dfrac{(L	 C_{P})^2(C_K)^4}{2^4\beta\epsilon_0}\vert	D\vert	t+\dfrac{C_F}{(\beta	\epsilon_0)^{1/3}}\Vert	F\Vert_{X^\prime}^{\frac{4}{3}}t.	\end{align}
	Therefore
	\begin{align}\label{est-invariant-final}
	\int_0^t\E\Vert  y(s)\Vert_{V}^2ds&\leq	\dfrac{C_P+1}{2\nu\epsilon_0} \Vert	y_0\Vert_2^2
	+\dfrac{C_P+1}{2\nu\epsilon_0} \big[\dfrac{(L	 C_{P})^2(C_K)^4}{2^4\beta\epsilon_0}\vert	D\vert	t+\dfrac{C_F}{(\beta	\epsilon_0)^{1/3}}\Vert	F\Vert_{X^\prime}^{\frac{4}{3}}t\big],\quad	\forall	t\in[0,T].	\end{align}
	Finally,	let	$R>0$	and	note	that,	after	using	Chebyshev	inequality	and	\eqref{est-invariant-final},	we	deduce
	\begin{align*}
	\dfrac{1}{T}\int_0^T(P_s^*\delta_{y_0})(H/\mathbb{B}_R)ds&=\dfrac{1}{T}\int_0^TP(\Vert	y(s)\Vert_V>R)ds\leq	\dfrac{1}{TR^2}\int_0^T\E\Vert	y(s)\Vert_V^2ds\\
	&\leq\dfrac{1}{TR^2}\dfrac{C_P+1}{2\nu\epsilon_0} \Vert	y_0\Vert_2^2
	+\dfrac{C_P+1}{R^2} \dfrac{L^2(	 C_{P})^2(C_K)^4}{2^5\beta\nu\epsilon_0^2}\vert	D\vert+\dfrac{1}{R^2}\dfrac{(C_P+1)C_F}{2\nu\epsilon_0(\beta	\epsilon_0)^{1/3}}\Vert	F\Vert_{X^\prime}^{\frac{4}{3}}.		
	\end{align*}
	Finally,	choose	any	$T_0>0$	and	$R:=R(\nu,\beta,\alpha,	C_K,C_P,C_F,L,\vert	D\vert,\Vert	y_0\Vert_2)	>0$	large	enough	to	obtain	\eqref{equ-invarinat}.	\end{proof}
Consequently,			Proposition	\ref{prop-inva-1},	Proposition		\ref{prop-invar-2},	\textit{"Krylov-Bogoliubov Theorem"}	(see	\textit{e.g.}	\cite[Theorem 3.1.1]{Daprato-ergodicity}	and		\cite[Corollary 3.1.2]{Daprato-ergodicity}	ensure	the	existence	of			an	invariant	measure	and	therefore	completes	the	proof	of	Theorem	\ref{THm-invariant-measure-1}.	Namely,
\begin{theorem}\label{THm-invariant-measure}	Assume that  $\overline{\mathcal{H}}_1$	holds.
	Then, there exists an invariant measure	$\mu\in\mathcal{P}(H)$,	the	set	of	Borel	probability	measure	on	$H$,	for	$(P_t)_t$	defined	by	\eqref{transition-prob}.	In	other	words,	$P_t^*\mu=\mu$		where	$(P_t^*)_t$	denotes	 the adjoint semi-group
	acting on	$\mathcal{P}(H)$	given	by
	\begin{align}\label{semi-group-adjoint}
	P_t^*\mu(\Gamma)=\int_HP_t(x,\Gamma)\mu(dx)	\text{	with	}	P_t(y_0,\Gamma):=P(u(t,y_0)\in\Gamma)	\text{	for	any	}	\Gamma\in\mathcal{B}(H).
	\end{align}
\end{theorem}

Let	us	present	the	following	concentration	property	of	the	invariant	measures	for	the	semigroup	$(P_t)_t$,	which	will	play	a	fundamental	role	to	prove	the	existence	of	an	ergodic	invariant	measure	in	Theorem	\ref{Theorem-ergodic}.	
\begin{prop}\label{prop-ergodic}
Let			$\mu$	be	an		invariant	measure	for	the	semigroup	$(P_t)_t$	defined	by	\eqref{transition-prob}.	Then
\begin{align}\label{prperty-invariant}
	\displaystyle\int_H\Vert	x\Vert_2^2\mu(dx)	\leq		
	\mathbf{K}\dfrac{C_P}{2\nu\epsilon_0}	\text{	and		}	\displaystyle\int_H\Vert	x\Vert_X^4\mu(dx)	\leq	\mathbf{K}\dfrac{\beta	\epsilon_0}{2(C_K)^4}[
	\dfrac{C_P}{2\nu\epsilon_0}+1],	
\end{align}	
where	$\mathbf{K}=\dfrac{(L	 C_{P})^2(C_K)^4}{2^4\beta\epsilon_0}\vert	D\vert	+\dfrac{C_F}{(\beta	\epsilon_0)^{1/3}}\Vert	F\Vert_{X^\prime}^{\frac{4}{3}}$,	$C_F>0$,	$C_P$	and	$C_K$	are	related	to	Poincaré	and	Korn	inequalities,	see	\eqref{poicare}	and	\eqref{Korn-ineq}.
\end{prop}
\begin{proof}
	Let	$y_0\in	H$	and	let	us	consider		the	following	function		$f_\epsilon:x\mapsto	\dfrac{x}{1+\epsilon	x},	\epsilon>0	$.	Note	that	$f_\epsilon\in		C^2_b(\mathbb{R}_+)$	satisfying
	\begin{align*}
x\in	\mathbb{R}_+:\quad	f_\epsilon^\prime(x)=\dfrac{1}{(1+\epsilon	x)^2}>0,	\quad	f_\epsilon^{\prime\prime}(x)=-\dfrac{2\epsilon}{(1+\epsilon	x)^3}<0.
	\end{align*}
Let	$\epsilon>0$,	by	applying	It\^o	formula	to	the	process	$\Vert	y\Vert_2^2$	given	by	\eqref{Ito-y}	and	the	function	$f_\epsilon$,	we	get
	\begin{align*}
	f_\epsilon(\Vert	y(t)\Vert_2^2)-f_\epsilon(\Vert	y_0\Vert_2^2)&=2\nu\int_0^tf_\epsilon^\prime(\Vert	y(s)\Vert_2^2)\langle	\Delta	y,y\rangle	ds+2\int_0^tf_\epsilon^\prime(\Vert	y(s)\Vert_2^2)\langle	\alpha\text{div}(A^2),y\rangle	ds\\
	&\quad+2\int_0^tf_\epsilon^\prime(\Vert	y(s)\Vert_2^2)\langle	\beta \text{div}(|A|^2A),y\rangle	ds+2\int_0^tf_\epsilon^\prime(\Vert	y(s)\Vert_2^2)\langle	F,y\rangle		ds\\&\quad+2\int_0^tf_\epsilon^\prime(\Vert	y(s)\Vert_2^2)(	G(\cdot,y),y)	d\mathcal{W}+	\int_0^tf_\epsilon^\prime(\Vert	y(s)\Vert_2^2)\sum_{\k\geq 1} \Vert	\sigma_{\k}(\cdot,y)\Vert_2^2ds\\
	&\quad+2\int_0^tf_\epsilon^{\prime\prime}(\Vert	y(s)\Vert_2^2)\sum_{\k\geq 1} \vert	(\sigma_{\k}(\cdot,y),y)\vert_2^2ds.
	\end{align*}
	Since	$f^{\prime\prime}_\epsilon<0$,	then	the	last	term	is	non	positive.	On	the	other	hand,
	by	using	that	$f^\prime_\epsilon>0$,		\eqref{est-invar-2}	and	\eqref{est-invar-3}	and	\eqref{exter-estim},	we	get
\begin{align*}
f_\epsilon(\Vert	y(t)\Vert_2^2)-f_\epsilon(\Vert	y_0\Vert_2^2)&\leq	-2\nu\epsilon_0\int_0^tf_\epsilon^\prime(\Vert	y(s)\Vert_2^2)\Vert \nabla y\Vert_{2}^2	ds-\dfrac{\beta	\epsilon_0}{2}\int_0^tf_\epsilon^\prime(\Vert	y(s)\Vert_2^2)	\int_D|A( y)|^4dxds\\&+		\dfrac{C_F}{(\beta	\epsilon_0)^{1/3}}t\Vert	F\Vert_{X^\prime}^{\frac{4}{3}}+2\int_0^tf_\epsilon^\prime(\Vert	y(s)\Vert_2^2)(	G(\cdot,y),y)	d\mathcal{W}+L\int_0^tf_\epsilon^\prime(\Vert	y(s)\Vert_2^2)\Vert	y(s)\Vert_2^2.
\end{align*}
	Recall	that	$f^\prime_\epsilon\leq	1$,	which	ensures	that		the	stochastic	integral	is	and	$(\mathcal{F}_t)$-martingale.	Hence,	by	taking	the	expectation	and	using	similair	arguments	for	\eqref{est-invariant2},	one	has
		\begin{align*}
	&\E	f_\epsilon(\Vert	y(t)\Vert_2^2)-f_\epsilon(\Vert	y_0\Vert_2^2)+2\nu\epsilon_0 \int_0^t\E	f_\epsilon^\prime(\Vert	y(s)\Vert_2^2)\Vert \nabla y\Vert_{2}^2ds\\
	&\leq	
	(LC_{P}(C_K)^4\dfrac{\theta}{2^5}-\dfrac{\beta	\epsilon_0}{2})\int_0^t\E	f_\epsilon^\prime(\Vert	y(s)\Vert_2^2)\int_D|A( y)|^4dxds+L	 C_{P}\dfrac{1}{2\theta}\vert	D\vert	t+\dfrac{C_F}{(\beta	\epsilon_0)^{1/3}}t\Vert	F\Vert_{X^\prime}^{\frac{4}{3}}.	\end{align*}
	Choose		$\theta=\dfrac{2^3\beta\epsilon_0}{LC_P(C_K)^4}$	and	using	\eqref{poicare},	we	obtain	
	\begin{align*}
	\E	f_\epsilon(\Vert	y(t)\Vert_2^2)&+\dfrac{2\nu\epsilon_0}{C_P} \int_0^t\E	f_\epsilon^\prime(\Vert	y(s)\Vert_2^2)\Vert  y(s)\Vert_{2}^2ds+\dfrac{\beta	\epsilon_0}{4}\int_0^tf_\epsilon^\prime(\Vert	y(s)\Vert_2^2)\E\int_D|A( y)|^4dxds\\&\quad	\leq	f_\epsilon(\Vert	y_0\Vert_2^2)
	+\dfrac{(L	 C_{P})^2(C_K)^4}{2^4\beta\epsilon_0}\vert	D\vert	t+\dfrac{C_F}{(\beta	\epsilon_0)^{1/3}}\Vert	F\Vert_{X^\prime}^{\frac{4}{3}}t\leq	f_\epsilon(\Vert	y_0\Vert_2^2)
	+\mathbf{K}t,	\end{align*}
	where	$\mathbf{K}=\dfrac{(L	 C_{P})^2(C_K)^4}{2^4\beta\epsilon_0}\vert	D\vert	+\dfrac{C_F}{(\beta	\epsilon_0)^{1/3}}\Vert	F\Vert_{X^\prime}^{\frac{4}{3}}$.	Therefore
	\begin{align*}
	\E	f_\epsilon(\Vert	y(t)\Vert_2^2)+\dfrac{2\nu\epsilon_0}{C_P} \int_0^t\E	\dfrac{\Vert	y(s)\Vert_2^2}{(1+\epsilon	\Vert	y(s)\Vert_2^2)^2}	ds	\leq	f_\epsilon(\Vert	y_0\Vert_2^2)
	+\mathbf{K}t,\quad	\epsilon>0.	\end{align*}
For	$y\in	H$,	set	$F_\epsilon(y)=f_\epsilon\circ	\Vert	y\Vert_2^2	$	and	note	that	$F_\epsilon\in	C_b(H)$.	We	recall	that	$P_tF_\epsilon(y_0)=\E	F_\epsilon(y(t))$.	Let	$\mu$	be	an	invariant	measure	for	$(P_t)_t$,	by	the	definition	of	invariant	measure	for	the	semigroup	$(P_t)_t$,	we	obtain	after	integrating	with	respect	to	$\mu$
	\begin{align*}
	\dfrac{2\nu\epsilon_0}{C_P} \int_H\int_0^t\E	\dfrac{\Vert	y(s)\Vert_2^2}{(1+\epsilon	\Vert	y(s)\Vert_2^2)^2}	ds	d\mu	\leq	
	\mathbf{K}t.	\end{align*}
	Let	$g_\epsilon(x)=\dfrac{x}{(1+\epsilon	x)^2}$	and	$G_\epsilon=g_\epsilon\circ\Vert	\cdot\Vert_2^2\in	C_b(H)$.	Hence	$\E	\dfrac{\Vert	y(s)\Vert_2^2}{(1+\epsilon	\Vert	y(s)\Vert_2^2)^2}:=P_sG_\epsilon(y_0).$
		Tonelli	theorem	and	the	invariance	of	$\mu$	ensure
	\begin{align*}
\dfrac{2\nu\epsilon_0}{C_P} \int_H\int_0^t\E	\dfrac{\Vert	y(s)\Vert_2^2}{(1+\epsilon	\Vert	y(s)\Vert_2^2)^2}	ds	d\mu=\dfrac{2\nu\epsilon_0}{C_P} \int_0^t\int_H	P_sG_\epsilon(y_0)d\mu	ds=t\dfrac{2\nu\epsilon_0}{C_P} \int_H	G_\epsilon(y_0)d\mu		\leq	
\mathbf{K}t.	\end{align*}
Finally,	by	letting	$\epsilon\to	0$	and	using	monotone	convergence	theorem	we	get

\begin{align}\label{prperty-invariant-1}
 \int_H	\Vert	y\Vert_2^2\mu(dy)		\leq	
\mathbf{K}\dfrac{C_P}{2\nu\epsilon_0}.	\end{align}
Next,	we	use	the	last	inequality	\eqref{prperty-invariant-1}	to	show
	the	second	inequality	in	\eqref{prperty-invariant}.	Indeed,	define	the	following	non	decreasing	sequence	
\begin{align}\label{sequ-appro-invari}
n\in	\mathbb{N}:\quad	F_n:H\to	\mathbb{R}_+\cup\{+\infty\};\quad	u\mapsto		\begin{cases}	\Vert	u\Vert_X^4	\text{	if	}	\Vert	u\Vert_X\leq	n;
\\[0.1cm]
n^4	\text{	else.		}
\end{cases}
\end{align}
and	note	that	$F_n$	converges	to	$F_X:=\displaystyle\sup_{n}F_n$,	where	
\begin{align*}
	F_X:H\to	\mathbb{R}_+\cup\{+\infty\};\quad	u\mapsto		\begin{cases}	\Vert	u\Vert_X^4	\text{	if	}		u\in	X;
\\[0.1cm]	
+\infty	\text{	if	}	u\in	H\setminus	X.
\end{cases}
\end{align*}	
	It	is	clear		that	$F_n\in	B_b(H)$\footnote{$B_b(H)$	denotes	the	set	of	bounded	Borel	functions	on	$H$.}	for	every	$n\in	\mathbb{N}$	and	$	F_n(u)\leq	\Vert	u\Vert_X^4$.	By	using	the	invariance	of	$\mu$,	we	are	able	to	infer
\begin{align}
\int_HF_nd\mu=\int_0^T\int_H\E	F_n(y(s))d\mu	ds=\int_H\int_0^T\E	F_n(y(s))	dsd\mu,
\end{align}
From	\eqref{concentration-eqn}	and	by	using	\eqref{Korn-ineq},	we	have
\begin{align}
\dfrac{\beta	\epsilon_0}{2(C_K)^4}\int_0^T\E	\Vert	y(s)\Vert_X^4ds\leq	\dfrac{\beta	\epsilon_0}{2}\int_0^T\E\int_D|A( y)|^4dxds\leq	\Vert	y_0\Vert_2^2
+\mathbf{K}T.	\end{align}
Thus	$
\displaystyle\int_0^T\E	F_n(y(s))	ds	\leq	\int_0^T\E	\Vert	y(s)\Vert_X^4	ds	\leq	\dfrac{\beta	\epsilon_0}{2(C_K)^4}[\Vert	y_0\Vert_2^2
+\mathbf{K}T].
$
Set	$T=1$	and		integrate	with	respect	to	$\mu$	the	last	inequality,	one	has
\begin{align*}
\int_HF_n(y_0)		d\mu=\int_H\int_0^1P_sF_n(y_0)	ds	d\mu=\int_H\int_0^1\E	F_n(y(s))	ds	d\mu	\leq	\dfrac{\beta	\epsilon_0}{2(C_K)^4}[\int_H\Vert	y_0\Vert_2^2d\mu
+\mathbf{K}].
\end{align*}
	Consequently,	the	monotone	convergence	theorem		and	\eqref{prperty-invariant-1}	imply
\begin{align*}
\int_HF_X(y_0)		d\mu	\leq	\mathbf{K}\dfrac{\beta	\epsilon_0}{2(C_K)^4}[
\dfrac{C_P}{2\nu\epsilon_0}+1].
\end{align*}
In	particular,	$\mu$	is	concentrated	on	$X$	and	$\mu(X)=1$.
\end{proof}
Recall	that		an	invariant	measure	$\mu$	is	ergodic	if	$$	\displaystyle\lim_{T\to	+\infty}\dfrac{1}{T}\int_0^TP_t\varphi	dt=\int_H\varphi(x)	\mu(dx),\quad		\forall	\varphi\in	L^2(H,\mu),$$
	see	\cite[Chapter	9]{Daprato-lecture}.	Thus,
it	follows	from		Proposition	\ref{prop-ergodic}	the	following	result.
\begin{theorem}\label{Theorem-ergodic}
There exists an ergodic invariant measure for the transition semigroup	$(P_t)_t$.
\end{theorem}
\begin{proof}

Denote	by	$	\Lambda$,		the	set	of	all	invariant	measures	for	the	Markov	semigroup	$(P_t)_t$	defined	by		\eqref{transition-prob}.		From	Theorem	\ref{THm-invariant-measure},
$\Lambda$	is	nonempty	convex	subset	of	$(C_b(H))^\prime$	and	\eqref{prperty-invariant}		ensures	that	$\Lambda$	is	tight,	since	$X\underset{compact}{\hookrightarrow}	H$.	Therefore,	Krein–Milman theorem		ensures	that		the	set	of	extreme points	is	non	empty	and	then	any	 extremal point of $\Lambda$ is an ergodic 	invariant	measure,	since
the	set	of	all	invariant	ergodic	measures	of	$(P_t)_t$	coincides	with	the	set	of	all	extremal	points	of	$\Lambda$,	see	\textit{e.g.}	\cite[Theorem	5.18]{Daprato-lecture}.
\end{proof}

\subsection*{Acknowledgements}
	This work is funded by national funds through the FCT - Funda\c c\~ao para a Ci\^encia e a Tecnologia, I.P., under the scope of the projects UIDB/00297/2020 and UIDP/00297/2020 (Center for Mathematics and Applications).



\end{document}